\documentclass{siamart250211}

\usepackage{amsmath}
\usepackage{amssymb}

\newsiamremark{remark}{Remark}

\headers{Refined stability estimates for
mixed problems}{N. Gauger, A. Linke, and C.Merdon}

\title{Refined stability estimates for
mixed problems by exploiting semi norm arguments\thanks{Submitted to the editors June 13, 2025.}}

\author{Nicolas Gauger\thanks{University of Kaiserslautern-Landau (RPTU),
Bldg/Geb 34, Paul-Ehrlich-Strasse, 67663 Kaiserslautern,
Germany (nicolas.gauger@rptu.de, alexander.linke@rptu.de)}
\and Alexander Linke\footnotemark[2]
\and Christian Merdon\thanks{Weierstrass Institute for Applied Analysis and Stochastics, 10117, Berlin, Germany (christian.merdon@wias-berlin.de)}}

\DeclareMathOperator{\Ker}{Ker}
\DeclareMathOperator{\Range}{Range}

\newcommand{\bbR}{\mathbb{R}}
\newcommand{\vecb}[1]{\mathbf{#1}}
\newcommand{\bu}{\vecb{u}}
\newcommand{\bv}{\vecb{v}}
\newcommand{\bw}{\vecb{w}}
\newcommand{\bfe}{\vecb{f}}
\newcommand{\bx}{\vecb{x}}
\newcommand{\bzero}{\vecb{0}}
\newcommand{\bL}{\vecb{L}}
\newcommand{\bH}{\vecb{H}}
\newcommand{\bomega}{\boldsymbol{\omega}}

\begin{document}

\maketitle

\begin{abstract} 
Refined stability estimates are derived for classical mixed problems.
The novel emphasis is on the importance of semi norms
on data functionals, inspired by recent progress on
pressure-robust discretizations
for the incompressible Navier--Stokes equations.
In fact, kernels of these semi norms are shown to be connected to physical regimes in applications
and are related to some well-known consistency errors in
classical discretizations of mixed problems.
Consequently, significantly sharper
stability estimates for solutions close to these physical
regimes are obtained.
%The contribution is intended as a first step towards
%a conceptual update of the classical theory of mixed
%finite element discretizations.
\end{abstract}

% Example of section
\section{Introduction}
In this contribution, we revisit the stability estimates
for mixed problems taught in classical and modern finite
element textbooks \cite{GR86, BF91, BBF, EGII,MR4713490} and suggest a fundamental refinement. The important implications of
that refinement are investigated in three kinds of mixed problems: finite dimensional (matrix) problems, infinite dimensional problems like the Helmholtz--Hodge decomposition or the Stokes problem and also spatial
discretizations of infinite dimensional mixed problems
by (finite dimensional) mixed finite element methods.
The emphasis is on the significance of semi norms,
which is - at least in mixed finite element methods - hardly discussed in the classical literature, and
its importance has been revealed only in recent years in the context
of so-called \emph{pressure-robust} space discretizations for the incompressible (Navier--)Stokes equations \cite{LINKE2014782,MR3564690, MR3683678, MR4031577, MR3875918}. Thus, this contribution
aims at generalizing this recent insight to a broader class of
mixed problems, yielding a first step towards a conceptual update of the classical theory of mixed finite element methods.

The emergence of semi norms in mixed problems can be made plausible
briefly:
For bounded linear functionals $f \in V^\prime$ and $g \in Q^\prime$ with Hilbert spaces $V$ and $Q$,
consider a mixed problem of the form: seek $(u,p) \in V \times Q$ such that
\begin{align}
  a(u,v) + b(v,p) & = \langle f, v \rangle_{V^\prime \times V} & \text{for all } v \in V, \label{eqn:model_problem} \\
           b(u,q) & = \langle g, q \rangle_{Q^\prime \times Q} & \text{for all } q \in Q, \nonumber
\end{align}
where $a$ and $b$ represent bounded bilinear forms.
Assuming -- as usual -- \emph{well-posedness} of the mixed problem,
we observe immediately two (unique) special solutions with $g=0$ for
\eqref{eqn:model_problem}, namely:
\begin{align*}
  p_0 \in Q \quad  \wedge \quad \langle f, \bullet \rangle_{V^\prime \times V} = b(\bullet, p_0) & \quad \Rightarrow \quad \left ( u, p \right ) = \left (0, p_0 \right ),
   \\
  u_0 \in V \quad  \wedge \quad \langle f, \bullet \rangle_{V^\prime \times V} = a(u_0, \bullet) & \quad \Rightarrow \quad \left ( u, p \right ) = \left (u_0, 0 \right ).
\end{align*}
Thus, it is straight-forward to conclude for the solutions $(u_1, p_1)$
and $(u_2, p_2)$ of the general mixed problem with right hand sides
$(f_1, g) \in V^\prime \times Q^\prime$
and $(f_2, g) \in V^\prime \times Q^\prime$:
\begin{align*}
  p_0 \in Q \  \wedge \ f_1 - f_2 = b(\bullet, p_0) & \quad \Rightarrow \quad u_1 = u_2 \ \left ( \wedge \ p_1 - p_2 = p_0 \right ), \quad (\text{$u$-equivalence}) \\
  u_0 \in V \  \wedge \ f_1 - f_2 = a(u_0, \bullet) & \quad \Rightarrow \quad p_1 = p_2 \ \left ( \wedge \ u_1 - u_2 = u_0 \right ), \quad (\text{$p$-equivalence}).
\end{align*}
Therefore, the mixed problem \eqref{eqn:model_problem}
induces two sets of equivalence classes on the data $f$, which
we call $u$-equivalence and $p$-equivalence. This reflects the fact
that only certain parts of the data $f$ influence each
of the solution parts $u$ and $p$. If we further want to quantify
precisely these contributions of $f$ on $u$ and $p$,
we are led to semi norms on the data $f$, since equivalence classes
are connected to semi norms. Thus, we present novel refined stability estimates for mixed problems that are more precise with respect to data
dependencies than the classical estimates, presented, e.g., in \cite{BF91, BBF}.
Note that the classical estimates also long for optimality, but their
focus is only on the best-possible dependencies on quantifications of the
\emph{involved operators} $a$ and $b$ in \eqref{eqn:model_problem}. We remark that applications
of an investigation based on semi norms on the data
$f$, are not limited to stability estimates. 
Indeed, appropriate semi norms in a-priori and a-posteriori error analysis 
have been recently derived
in the context of pressure-robust finite element methods for the incompressible Stokes problem. On one hand, this led to
sharper data-dependent a priori error estimates, see in particular \cite[Section~4]{MR3683678}, \cite[Section~7]{ETNA2020} and \cite{MR3564690,MR3787384,BEIRAODAVEIGA2022115061,ZZ2025} for details, and on the other hand to more efficient adaptive refinement algorithms
through computable approximations of these semi norms, see 
in particular \cite[Section~5]{NuMa2019} and \cite{LM2015, JNM2022} for details.
Here, pressure-robustness for the incompressible Stokes problem
corresponds to $u$-equivalence in the mixed problem \eqref{eqn:model_problem}.
In contrast, the notion $p$-equivalence is used in this contribution for the
first time, but \cite{MR3407240}
derived a Stokes discretization
fulfilling such a property.

The significance of semi norms remains
if we change the model problem under investigation. E.g., $u$-equivalence
is invariant under changes of $a(\bullet, \bullet)$, since it depends only
on the bilinear form $b(\bullet, \bullet)$. As an example, we present below
how --- \emph{using semi norm arguments} --- some conclusions on the incompressible Stokes problem can be extended to
the (time-dependent) and nonlinear incompressible Navier--Stokes equations.
These considerations motivated the development of a novel pressure-robust convection stabilization in \cite{ABBGLM}, the first optimally converging stabilization for the divergence-free Scott--Vogelius finite element method.
Furthermore, semi norms are connected to
\emph{physical regimes} like \emph{hydrostatics},
\emph{potential flows}, \emph{vortex-dominated flows} and \emph{geostrophic flows}.
Such observations on hydrostatics and geostrophic flows
motivated the use of certain
mixed finite element methods, ---
which are pressure-robust when applied to the incompressible
Stokes problem, --- for numerical weather prediction \cite{MR2969702}.
Our exposition may serve as a blueprint to connect physical regimes via semi norms in other mixed problems.

Last but not least, examples of important finite-dimensional, infinite-dimensional and discretized mixed problems are presented in more detail.
The refined stability estimates
are shown to be more accurate than the classical estimates, in general.
Further, the intrinsic stability estimates of two classical discretization
schemes for the incompressible Stokes problem are compared, and
their different qualitative behavior is connected to a consistency error
with respect to $u$-equivalence, introduced above.

%Here, the main observation is that
%force balances like $\nabla p = \bfe$,
%$\nabla \cdot \left (\bu \otimes \bu \right ) + \nabla p = \bzero$, 
%$2 \boldsymbol{\Omega} \times \bu + \nabla p = \bzero$
%lead to different physical regimes
%for incompressible flows, whenever
%$\bfe$, the nonlinear convection term or
%the Coriolis force can be represented by
%a gradient field $\nabla \phi$.

The paper is structured as follows. Section~\ref{sec:refined_estimates} presents the
refined stability estimates for quite general
nonsymmetric and symmetric frameworks in Hilbert spaces.
Section~\ref{sec:examples} studies analytical examples
that demonstrate the qualitative improvements
of the refined estimates for finite-dimensional and
infinite-dimensional mixed problems.
Section~\ref{sec:Inhom:NSE} explains the relevance of semi norm
arguments for the establishment of physical regimes in physical models,
which extend linear mixed problems, like the incompressible
Navier--Stokes equations.
Last but not least, Section~\ref{sec:pressure:robustness}
emphasizes the importance of semi norm arguments in
finite dimensional approximations (by finite element methods)
of infinite dimensional problems.
Section~\ref{sec:conclusions}
draws some conclusions.

%%%%%%%%%%%%%%%%%%%%%%%%%%%%
\section{Refined stability estimates
for mixed methods}
\label{sec:refined_estimates}
This section revisits the mixed methods theory from \cite{BF91, BBF} for problem
\eqref{eqn:model_problem} to refine it in certain
aspects.

%%%%%%%%%%%%%%%%%%%%%%%%%%%%%%%%%%%%%%%
\subsection{Notation and general assumptions}
The usual operator norm for $f \in V^\prime$ is used, i.e.,
\begin{align}
   \| f \|_{V^\prime} & := \sup_{v \in V \setminus \lbrace 0\rbrace} \frac{\left | \langle f, v \rangle_{V^\prime \times V} \right |}{\| v \|_V}. \label{norm:functional}
\end{align}
We also introduce an operator norm, where the supremum in \eqref{norm:functional}
is taken over a closed
subspace $U \subset V$, where we will denote the
corresponding expression as
\begin{align}
   \| f \|_{U^\prime} & := \sup_{v \in U \setminus \lbrace 0\rbrace} \frac{\left | \langle f, v \rangle_{V^\prime \times V} \right |}{\| v \|_V}. \label{seminorm:functional}
\end{align}
We emphasize that \eqref{seminorm:functional} represents a
\emph{semi norm}, in general.
The bilinear forms $a : V \times V \to \bbR$ and $b : V \times Q \to \bbR$ are assumed to be continuous, i.e.,
\begin{align}
  |a(u,v)| \leq \| a \| \|u\|_V \|v\|_V\label{eqn:a_continuity}
  \quad \text{and} \quad
  |b(v,q)| \leq \| b \| \|v \|_V \|q\|_Q.
\end{align}
Thus, the bilinear form $b$ defines two operators $B : V \to Q^\prime$ and $B^t : Q \to V^\prime$ by
\begin{align*}
  \langle B v, q \rangle_{Q^\prime \times Q} = \langle v, B^t q \rangle_{V \times V^\prime} = b(v,q)
  \quad \text{for all } v \in V, q \in Q,
\end{align*}
and the bilinear form $a$ defines an operator $A : V \to V^\prime$ by
\begin{align*}
  \langle A v, w \rangle_{V\prime \times V} = a(v, w)
  \quad \text{for all } v, w \in V.
\end{align*}
The model problem \eqref{eqn:model_problem}
can be written as an equivalent operator problem
\begin{align}
  A u + B^t p & = f \quad \text{in $V^\prime$}, \\
  B u & = g \quad \text{in $Q^\prime$}.
\end{align}
Important subspaces for the problem
under consideration are defined by
\begin{align}
  K & := \Ker B =
  \left \{ v \in V : b(v, q) = 0 \text{ for all } q \in Q \right \} \subset V, \\
  K^0 & := \left \{ f \in V^\prime:
   \langle f, v^0 \rangle_{V^\prime \times V} = 0 \text{ for all $v^0 \in K$}
   \right \}.
\end{align}
Both subspaces $K \subset V$ and $K^0 \subset V^\prime$ are closed, see, e.g.,
\cite[Theorem 4.1.5]{BBF}.
$K^0$ is called the polar space of $K$.
In this contribution, we
assume $\Range B = Q^\prime$ and that
$a(\bullet, \bullet)$ satisfies the
inf-sup condition
\begin{align}
 \label{eqn:a_kernel_infsup1}
   0 < \alpha_0 \leq \inf_{0 \not= u \in K} \sup_{0 \not= v \in K} \frac{a(u,v)}{\| u \|_V \| v \|_V}
 %\label{eqn:a_kernel_infsup2}
 %  \inf_{0 \not= v \in K} \sup_{0 \not= u \in K} \frac{a(u,v)}{\| u \|_V \| v \|_V} & \geq \alpha_2 > 0.
\end{align}
and the non-degeneracy condition
\begin{align}\label{eqn:non_degeneracy}
   u \in K \quad \& \quad a(u,v) = 0 \text{ for all } v \in K \quad \Rightarrow \quad u = 0.
\end{align}
\begin{remark}
It is well known, that if $a$ is coercive on $K$, i.e.,
\begin{equation}
    \exists_{\alpha > 0} :
    a(v^0, v^0) \geq \alpha \| v^0 \|_V^2
    \quad \text{for all $v^0 \in K$},
\end{equation}
conditions \eqref{eqn:non_degeneracy} and \eqref{eqn:a_kernel_infsup1} 
are satisfied with $\alpha_0 \geq \alpha$. 
\end{remark}
The surjectivity of $B$ implies that there
is a $\beta > 0$, such that
for all $g \in Q^\prime$ there exists
a $v \in V$ with $B v = g$ and
\begin{equation}\label{eqn:infsup}
  \| v \|_V \leq \frac{1}{\beta}
  \| g \|_{Q^\prime}.
\end{equation}
By the Banach closed range theorem
\cite[Theorem 4.1.5]{BBF} this is equivalent
to the boundedness of $B^t$ in the sense
\begin{align}\label{eqn:infsup2}
  0 < \beta \leq \inf_{q \in Q \setminus \lbrace 0 \rbrace}
  \sup_{v \in V \setminus \lbrace 0 \rbrace}
  \frac{b(v,q)}{\| v \|_V \|q\|_Q}.
\end{align}

%%%%%%%%%%%%%%%%%%%%%%%%%%%%%%%%%%%%%%%%
\subsection{Refined stability estimates - general case}
We are now able to state our first main
result, namely sharper stability estimates than those given in classical and modern textbooks \cite[Proposition II.1.3]{BF91},
\cite[Theorem 4.2.3]{BBF}, \cite[Theorem 49.13]{EGII}, 
\cite[Theorem 1.3.27]{BGHRR24}, \cite[Theorem 4.3.1]{MR4713490}, \cite[Chapter 12.2]{MR2373954}
and \cite[Corollary 4.1, p. 61]{GR86}.
Note that in \cite{GR86, MR2373954} it is only
stated that the problem is well-posed, but
no stability estimates are given. Compared to the other mentioned references the stability estimates are refined with respect to semi norms of the data $f$.
\begin{theorem}\label{thm:stability_general}
The solution $(u,p) \in V \times Q$ of \eqref{eqn:model_problem} is unique and satisfies
\begin{align*}
  \| u \|_V & \leq
  \frac{1}{\alpha_0}  \| f \|_{K^\prime} + \frac{1}{\beta} 
               \left ( 1 +\frac{\| a \|}{\alpha_0} \right ) \| g\|_{Q^\prime}, \\
%  || p||_Q
%    & \leq {\color{gray} \frac{1}{\beta} ||f||_{\left(K^\perp\right)^\prime} + \frac{1}{\beta} \frac{||a||}{\alpha} ||f||_{K^\prime}
 %     + \frac{||a||}{\beta^2}
 %     \left (1+ \frac{||a||}{\alpha} %\right ) ||g||_{Q^\prime}}\\
 \| p\|_Q
    & \leq \frac{1}{\beta}  \inf_{w^0 \in K} \left(\| f - a(w^0,\bullet) \|_{V^\prime} + \frac{\| a \|}{\alpha_0} \| f - a(w^0,\bullet) \|_{K^\prime} \right) 
      + \frac{\| a \|}{\beta^2} \! \left ( 1 + \frac{\|a\|}{\alpha_0}\right ) \!
        \| g \|_{Q^\prime}.
\end{align*}
\end{theorem}

The general structure of the proof follows the proof of \cite[Theorem 4.2.3]{BBF}.
\begin{proof}[Proof of existence of (u,p)]
From the surjectivity of $B$ follows the existence of a lifting operator $L_B : Q^\prime \rightarrow V$
such that $B(L_B g) = g$ for all $g \in Q^\prime$. Setting $u^g := L_B g$, we therefore have $B u^g=g$,
and we want to consider the new unknown $u^0 := u-u^g \in K$.
For any $v^0 \in K$ it holds $b(v^0, q)=0$ for all $q \in Q$ and therefore
\begin{align}\label{eqn:thm1_intermediate_estimate1}
  a(u^0, v^0) = \langle f, v^0 \rangle_{V^\prime \times V} - a(u^g, v^0) \in K^\prime.
\end{align}
Condition \eqref{eqn:non_degeneracy} yields
the existence of (a unique) $u^0 \in K$ by the Babu\v{s}ka--Lax--Milgram theorem.
For the functional $l \in V^\prime$ defined by
\begin{align}\label{eqn:definition_l}
\langle l, v \rangle_{V^\prime \times V} := \langle f, v \rangle_{V^\prime \times V} - a(u^0 + u^g, v) \quad \text{for all } v \in V,
\end{align}
\eqref{eqn:thm1_intermediate_estimate1} implies $l(v^0) = 0$ for all $v^0 \in K$, thus
$l \in K^0$. However, $K^0$ coincides with $\Range B^t$,
according to the Banach closed range theorem
\cite[Theorem 4.1.5]{BBF}. Hence,
$l$ is in the image of $B^t$ and there exists a $p \in Q$ with $B^t p = l$.
Therefore, it holds
$$
  \langle B^t p, v \rangle_{V^\prime \times V} = \langle l, v \rangle_{V^\prime \times V}
    = \langle f, v \rangle_{V' \times V} - a(u^0 + u^g, v)
    \quad \text{for all } v \in V,
$$
and $(u^0 + u^g, p) \in V \times Q$ fulfills the first equation
of \eqref{eqn:model_problem}. Further, it holds
$B u = B u^0 + B u^g = g$ and the second equation of \eqref{eqn:model_problem} is also fulfilled.
\end{proof}
\begin{proof}[Proof of uniqueness of (u,p)]
To prove uniqueness, one exploits the linearity of the problem. For data
$f=0$ and $g=0$, it holds $u \in K$. Therefore, it holds
$$
  a(u, v) + b(v, p) = a(u, v) = \langle 0, v  \rangle_{V' \times V} = 0
  \quad \text{for all } v \in K
$$
and one concludes $u=0$ due to \eqref{eqn:non_degeneracy}. Consequently, it holds
$$
  b(v, p) = 0
  \quad \text{for all } v \in V.
$$
Due to \cite[Corollary 4.1.1]{BBF} $B^t$ is injective and thus $p=0$ is the unique solution.
\end{proof}
\begin{proof}[Invariance of $p$ with respect to $A(K)$]
  Due to the uniqueness of the solution and the linearity of the problem
  \eqref{eqn:model_problem}, we observe that subtracting any
  $a(w^0,\bullet) \in A(K)$ from the right-hand side of the equation
  only changes $u$ to $u-w^0$ but not $p$. Hence, for estimating
  $\| p \|_Q$ below it is allowed to subtract any $a(w^0,\bullet) \in A(K)$.
\end{proof}
\begin{proof}[Proof of the stability bounds]
Recall that $\| u^g\|_V \leq \beta^{-1} \| g \|_{Q^\prime}$ due to \eqref{eqn:infsup}.
Further, it holds
\begin{align*}
  \alpha_0 \| u^0 \|_V \leq \sup_{0 \not= v \in K} \frac{a(u^0, v)}{\| v \|_V} & = \sup_{0 \not= v \in K} \frac{\langle f, v \rangle_{V^\prime \times V} - a(u^g, v)}{\| v \|_V}
     \leq \| f\|_{K^\prime} + \|a\| \, \| u^g \|_V.
\end{align*}
Thus,
\begin{align*}
  \| u \|_V \leq \| u^0 \|_V + \| u^g \|_V
               & \leq \frac{1}{\alpha_0}  \| f \|_{K^\prime} + \left (1 + \frac{\| a \|}{\alpha_0} \right )
                 \| u^g \|_V \\
               & \leq \frac{1}{\alpha_0}  \| f \|_{K^\prime} + \frac{1}{\beta} 
               \left ( 1 +\frac{\|a\|}{\alpha_0} \right ) \| g \|_{Q\prime}.
\end{align*}
Concerning the bound for $p$, \eqref{eqn:definition_l}, \eqref{eqn:a_continuity}, and the already proven estimate for $u$ yield
\begin{align*}
  \| p \|_Q \leq \frac{1}{\beta} \| l \|_{V^\prime}
    & = \frac{1}{\beta} \| f - a(u, \bullet) \|_{V^\prime}\\
    & \leq  \frac{1}{\beta} \left ( \| f \|_{V^\prime} + \| a \| \cdot \| u \| \right ) \\
    & \leq \frac{1}{\beta} \left( \| f \|_{V^\prime}
      + \frac{\| a \|}{\alpha_0} \| f \|_{K^\prime} \right) 
      + \frac{\| a \|}{\beta^2} \left ( 1 + \frac{\| a \|}{\alpha_0} \right )
        \| g \||_{Q^\prime}.
\end{align*}
Since $p$ is invariant against
any change of
$f$ to $f - a(w^0, \bullet)$ for $w^0 \in K$, one can take the infimum as stated in the assertion.
This concludes the proof of Theorem~\ref{thm:stability_general}.
\end{proof}

\begin{remark}[$u$/$p$-equivalence] \label{rem:compliance:equivalence:gen}
Theorem \ref{thm:stability_general} complies with the $u$-equivalence
and the $p$-equivalence of the mixed problem. Indeed,
for any $q \in Q$ setting
$\bfe = b(\bullet, q)$ and $g=0$, yields
$$
  \| u \|_V  \leq
  \frac{1}{\alpha_0}  \| b(\bullet, q) \|_{K^\prime} = 0
  \quad \text{and hence $u$-equivalence}.
$$
Further, for any $v^0 \in K$
setting $f = a(v^0, \bullet)$ and $g=0$ yields
$$
\| p\|_Q
     \leq \frac{1}{\beta}  \inf_{w^0 \in K} \left(\| a(v^0, \bullet) - a(w^0,\bullet) \|_{V^\prime} + \frac{\| a \|}{\alpha_0} \| a(v^0, \bullet) - a(w^0,\bullet) \|_{K^\prime} \right) = 0,
$$
since we can set $w^0 := v^0$ in the infimum, assuring $p$-equivalence.
\end{remark}
%%%%%%%%%%%%%%%%%%%%%%
\subsection{Refined estimates - symmetric case}
This section, derives further refined stability estimates under stronger assumptions on the bilinear
form $a(\bullet, \bullet)$, which hold, e.g.,
for the incompressible Stokes problem.
These additional assumptions are
\begin{align}
  a(u,v) & = a(v,u)  && \text{(symmetry),} \label{eqn:a_symmetry} \\
  \exists \alpha > 0, \quad a(u,u) & \geq \alpha \|u \|^2_V 
  \quad \text{for all } u \in V && \text{(coercivity)} \label{eqn:a_coercivity}.
\end{align}
Then, $a(\bullet, \bullet)$ defines a scalar product on $V \times V$, and there exists an
$a$-orthogonal decomposition of $V$ into the subspaces $K$ and
\begin{align*}
  K_a^\perp := \left\lbrace v^\perp \in V : a(v^\perp,v^0) = 0 \text{ for all } v^0 \in K \right\rbrace.
\end{align*}
The operators $\Pi_K : V \rightarrow K$ and $\Pi_{K_a^\perp} := 1 - \Pi_K$
denote the projectors onto $K$ and $K_a^\perp$, respectively, defined by $a(\Pi_{K}v - v, w^0) = 0$ for all $v \in V$ and for all $w^0 \in K$.
There is a possibly larger coercivity constant $\alpha_0$ restricted to functions from $K$, i.e.,
\begin{align}\label{eqn:a_coercivity_K}
\exists \alpha_0 \geq \alpha > 0, \quad a(u,u) & \geq \alpha_0 \|u \|^2_V
\quad \text{for all } u \in K.
\end{align}
The stronger assumptions allow for the following stability estimates.
%%%%%%%%%%%%%%%%%%%%%%%%%%%%%%%%%%%%%%%%%%%%%%
\begin{theorem}\label{thm:stability_symmetric}
Assuming
\eqref{eqn:a_symmetry} and
\eqref{eqn:a_coercivity},
  the solution $(u,p) \in V \times Q$ of \eqref{eqn:model_problem} satisfies
  \begin{align}
    \|u\|_V & \leq \frac{1}{\alpha_0} \| f \|_{K^\prime}\label{eqn:u_stability_symmetric} + \frac{1}{\beta} \left( \frac{\| a \|}{\alpha} \right)^{1/2} \| g \|_{Q^\prime},\\
    \|p\|_Q & \leq \frac{1}{\beta} 
\| f \circ \Pi_{K_a^\perp} \|_{V^\prime}
+ \frac{\| a \|}{\beta^2} \| g \|_{Q^\prime}
\leq
    \frac{1}{\beta} \left(\frac{\|a\|}{\alpha}\right)^{1/2} \| f \|_{(K_a^\perp)^\prime}
    + \frac{\| a \|}{\beta^2} \| g \|_{Q^\prime}.\label{eqn:p_stability_symmetric}
  \end{align}
\end{theorem}
\begin{proof}
  The solution is decomposed into $u = u_f + u_g$ and $p = p_f + p_g$ such that
  \begin{align}
    a(u_f,v) + b(v,p_f) & = \langle f, v \rangle_{V^\prime \times V} && \text{for all } v \in V\label{eqn:model_problem_V0},\\
               b(u_f,q) & = 0 && \text{for all } q \in Q,\nonumber\\
    a(u_g,v) + b(v,p_g) & = 0 && \text{for all } v \in V,\label{eqn:model_problem_VP}\\
               b(u_g,q) & = \langle g, q \rangle_{Q^\prime \times Q} && \text{for all } q \in Q.\nonumber
  \end{align}
  From the second equation in \eqref{eqn:model_problem_V0} and testing the first equations of \eqref{eqn:model_problem} and \eqref{eqn:model_problem_V0}
  with $v = v^0 \in K$ one obtains
  \begin{align*}
     u_f \in K \quad \text{and} \quad a(u - u_f, v^0) = 0
     \quad \text{for all } v^0 \in K.
  \end{align*}
  Hence, $u_f = \Pi_K u =: u^0 \in K$.
  Similarly, testing the first equation in \eqref{eqn:model_problem_VP} with $v = v^0 \in K$ yields $u_g \in K_a^\perp $ and testing the first equations in \eqref{eqn:model_problem}, \eqref{eqn:model_problem_V0} and \eqref{eqn:model_problem_VP} with $v = v^\perp \in K_a^\perp$ gives
  \begin{align*}
    a(u - u_g, v^\perp) + b(v^\perp, p - p_g) = \langle f, v^\perp \rangle_{V^\prime \times V} = b(v^\perp, p_f)
     \quad \text{for all } v^\perp \in K_a^\perp.
  \end{align*}
  Since $p - p_g = p_f$, it follows $a(u - u_g, v^\perp) = 0$ for all $v^\perp \in K_a^\perp$ and hence $u_g = \Pi_{K_a^\perp} u =: u^\perp \in K_a^\perp$.
  
  The norm of $u_f$ can be estimated by \eqref{eqn:a_coercivity_K} and testing \eqref{eqn:model_problem_V0} with $v = u_f$, i.e.,
  \begin{align}\label{eqn:estimate_u0}
    \alpha_0 \| u_f \|^2_V \leq a(u_f,u_f) = \langle f, u_f \rangle_{V^\prime \times V} \leq \| f \|_{K^\prime} \| u_f \|_{V}.
  \end{align}
  
 For all $v \in V$, the estimate
 \begin{equation}
       \| v^{\perp} \|_V \leq \left ( \frac{\| a \|}{\alpha} \right )^\frac{1}{2} \| v \|_V
       \label{eqn:perp_norm_estimate}
 \end{equation}
 is readily proved by
 \begin{align}
    \| v^{\perp} \|_V^2
    \leq \frac{1}{\alpha} a(v^{\perp}, v^{\perp})
    \leq \frac{1}{\alpha} \left( a(v^{\perp}, v^{\perp}) + a(v^{0},v^{0}) \right) = \frac{1}{\alpha} a(v,v)
    & \leq \frac{\| a \|}{\alpha} \| v \|_V^2.
\end{align}
 
 %This, the inf-sup condition \eqref{eqn:infsup2} and Equation \eqref{eqn:model_problem_V0} yield
 %\begin{align*}
 %   \beta \| p_f \|_Q 
 %   & \leq \sup_{v \in V \setminus \lbrace 0 \rbrace} %\frac{b(v,p_f)}{\| v \|_V}\\
 %   & \leq \left( \frac{\| a \|}{\alpha} \right)^{1/2} \sup_{v^\perp \in K_a^\perp \setminus \lbrace 0 \rbrace} \frac{b(v^\perp,p_f)}{\| v^\perp \|_V}\\
 %   & = \left( \frac{\| a \|}{\alpha} \right)^{1/2} \sup_{v^\perp \in K_a^\perp \setminus \lbrace 0 \rbrace} \frac{\langle f, v^\perp \rangle_{V^\prime \times V} - a(u_f, v^\perp)}{\| v^\perp \|_V}\\
 %   & = \left( \frac{\| a \|}{\alpha} \right)^{1/2} \| f \|_{(K_a^\perp)^\prime}
%\end{align*}
The inf-sup condition \eqref{eqn:infsup2} and Equation \eqref{eqn:model_problem_V0}
yield
\begin{align*}
\beta \| p_f \|_Q 
     \leq \sup_{v \in V \setminus \lbrace 0 \rbrace} \frac{b(v,p_f)}{\| v \|_V}
     = \sup_{v \in V \setminus \lbrace 0 \rbrace} \frac{\langle f, v \rangle_{V^\prime \times V} - a(u_f, v)}{\| v \|_V}
     = \| f - a(u_f, \bullet) \|_{V^\prime}.
\end{align*}
Let us prove the equivalence
\begin{align*}
\| f \|_{(K_a^\perp)^\prime}
\leq \| f - a(u_f, \bullet) \|_{V^\prime}
= \| f \circ \Pi_{K_a^\perp} \|_{V^\prime}
\leq \left( \frac{\| a \|}{\alpha} \right)^{1/2} \| f \|_{(K_a^\perp)^\prime}.
\end{align*}
Indeed, testing only with $v \in K_a^\perp$ gives
\begin{align*}
\| f - a(u_f, \bullet) \|_{V^\prime}
\geq \| f - a(u_f, \bullet) \|_{(K_a^\perp)^\prime}
= \| f \|_{(K_a^\perp)^\prime}.
\end{align*}
Since $u_f \in K$ satisfies \eqref{eqn:model_problem_V0},
together with \eqref{eqn:perp_norm_estimate} 
it follows
\begin{align*}
 \| f - a(u_f,\bullet) \|_{V^\prime}
& = \sup_{v \in V \setminus \lbrace 0 \rbrace} \frac{\langle f, v \rangle_{V^\prime \times V} - a(u_f, v)}{\| v \|_V}\\
& = \sup_{v \in V \setminus \lbrace 0 \rbrace} \frac{\langle f, v^\perp \rangle_{V^\prime \times V} - a(u_f, v^\perp)}{\| v \|_V}
= \| f \circ \Pi_{K_a^\perp} \|_{V^\prime}\\
& \leq \left( \frac{\| a \|}{\alpha} \right)^{1/2} \sup_{v^\perp \in K_a^\perp \setminus \lbrace 0 \rbrace} \frac{\langle f, v^\perp \rangle_{V^\prime \times V}}{\| v^\perp \|_V} \\
& = \left( \frac{\| a \|}{\alpha} \right)^{1/2} 
\| f \|_{(K_a^\perp)^\prime}.
\end{align*}
Eventually, one arrives at
\begin{align*}
\| p_f \|_Q \leq \frac{1}{\beta} 
\| f \circ \Pi_{K_a^\perp} \|_{V^\prime}
\leq \frac{1}{\beta} \left( \frac{\| a \|}{\alpha} \right)^{1/2} \| f \|_{(K_a^\perp)^\prime}.
\end{align*}

It remains to estimate the norms of $u_g$ and $p_g$. Testing \eqref{eqn:model_problem_VP} with $v = u_g$
yields
 \begin{align*}
     a(u_g,u_g) = -b(u_g,p_g) = -\langle g, p_g \rangle_{Q^\prime \times Q} \leq \| p_g \|_Q \|g \|_{Q^\prime}.
 \end{align*}
 The inf-sup condition \eqref{eqn:infsup2}, Equation \eqref{eqn:model_problem_VP} and a Cauchy--Schwarz inequality
 reveal
 \begin{align*}
    \beta \|p_g \|_Q
    \leq \sup_{v \in V \setminus \lbrace 0 \rbrace} \frac{b(v, p_g)}{\| v\|_V}
    = \sup_{v \in V \setminus \lbrace 0 \rbrace} \frac{-a(u_g, v)}{\| v\|_V}
    %& \leq \sup_{v \in V \setminus \lbrace 0 \rbrace} \frac{a(u_g, u_g)^{1/2} a(v,v)^{1/2}}{\| v\|_V}\\
    & \leq \|a\|^{1/2} a(u_g,u_g) ^{1/2}.
 \end{align*}
 The combination with the previous estimates yields
 \begin{align*}
     a(u_g,u_g)^{1/2} \leq \frac{\|a\|^{1/2}}{\beta} \|g \|_{Q^\prime} \quad \text{and} \quad
     \| p_g \|_Q^{1/2} \leq \frac{\|a\|^{1/2}}{\beta} \|g \|^{1/2}_{Q^\prime}
 \end{align*}
 and hence
 \begin{align*}
     \| u_g \|_V \leq \frac{1}{\alpha^{1/2}} a(u_g,u_g)^{1/2} \leq 
     \left ( \frac{\| a \|}{\alpha} \right )^{1/2} \frac{1}{\beta} \|g \|_{Q^\prime}
     \quad \text{and} \quad
     \| p_g \|_Q \leq \frac{\|a\|}{\beta^2} \|g \|_{Q^\prime}.
 \end{align*}
 The combination of the estimates of $u_g$ and $u_f$
 as well as $p_f$ and $p_g$
 and a triangle inequality yield \eqref{eqn:u_stability_symmetric} as well as \eqref{eqn:p_stability_symmetric}, respectively.
\end{proof}

% \begin{remark}[Alternative last step]
% By replacing the
% triangle inequality in the end by a Pythagoras identity the possibly slightly sharper estimate
% \begin{align*}
% \| u \|^2_V \leq \frac{1}{\alpha} \left(a(u_f,u_f) + a(u_g,u_g) \right)
% \leq \frac{1}{\alpha^2} \| f \|^2_{K^\prime} + \frac{1}{\beta^2} \left( \frac{\| a \|}{\alpha} \right) \| g \|^2_{Q^\prime}
% \end{align*}
% can be obtained, but with $1/\alpha$ instead of $1/\alpha_0$ in the first term.
% \end{remark}

\begin{remark}[$u$/$p$-equivalence] 
Similar to Remark \ref{rem:compliance:equivalence:gen}, it can be shown that
Theorem \ref{thm:stability_symmetric}
complies with the $u$-equivalence and the $p$-equivalence.
For the $p$-equivalence, any $w^0 \in K$ and $f := a(w^0, \bullet)$ imply
$f(v^\perp) = a(w^0, v^\perp) = 0$
for all $v^\perp \in K_{a^\perp}$.
\end{remark}

\begin{remark}[Comparison with \cite{BBF}]
In the symmetric case of \cite[Theorem 4.2.3]{BBF}
bounds with similar constants but without semi norms on the data are obtained, i.e.,
 \begin{align}
    \|u\|_V & \leq \frac{1}{\alpha_0} \| f \|_{V^\prime}\label{eqn:u_stability_classical} + \frac{2}{\beta} \left( \frac{\| a \|}{\alpha_0} \right)^{1/2} \| g \|_{Q^\prime},\\
    \|p\|_Q & \leq \frac{2}{\beta} 
    \left(\frac{\|a\|}{\alpha_0}\right)^{1/2} \| f \|_{V^\prime}
    + \frac{\| a \|}{\beta^2} \| g \|_{Q^\prime}.\label{eqn:p_stability_classical}
  \end{align}
Interestingly, \cite[Theorem 4.3.1]{BBF} derives
bounds with some semi norms for a perturbed saddle point problem with symmetric bilinearforms $a$ and
$c : Q \times Q \rightarrow \mathbb{R}$ with
\begin{align*}
  c(q,q) \geq \gamma \| q \|^2_Q.
\end{align*}
In the limit $\| c \| \rightarrow 0$, $\gamma \rightarrow \infty$ and $g_0 = 0$ one obtains
\begin{align}
    \|u\|_V & \leq \frac{2}{\alpha_0} \| f \|_{K^\prime}
    +\frac{1}{\beta}\left(1 + \frac{\|a\|^{1/2}}{\alpha_0^{1/2}}\right) \| g \|_{Q^\prime},\label{eqn:estimate_u_BBF_perturbed}\\
    \|p\|_Q & \leq \frac{1}{\beta} \| f \|_{(K^\perp)^\prime}
    + \frac{3}{\beta} \left(\frac{\|a\|}{\alpha_0}\right)^{1/2}\| f \|_{K^\prime}
    + \frac{\| a \|}{\beta^2} \| g \|_{Q^\prime}.\label{eqn:estimate_p_BBF_perturbed}
\end{align}
Here, $K^\perp$ is the orthogonal complement of $K$ in $V$ w.r.t.\ the standard scalar product in $V$ and not w.r.t.\ $a$ as in
Theorem~\ref{thm:stability_symmetric}.
The estimate \eqref{eqn:estimate_u_BBF_perturbed} for $u$ implies $u$-equivalence, but it is not mentioned by the authors. Moreover, the semi norms in \eqref{eqn:estimate_p_BBF_perturbed} do not allow to conclude $p$-equivalence, because
it also involves $\|f\|_{K^\prime}$ and $\| a(u^0,\bullet) \|_{(K^\perp)^\prime}$ for $u^0 \in K$ does not necessarily vanish. The main conclusion of \cite[Theorem 4.2.3]{BBF}
is the stability estimate of the form
\(
  \|u \|_V + \| p \|_Q \leq C \left(\| f \|_{V^\prime} + \| g \|_{Q^\prime} \right),
\)
which is derived for a product norm neither allowing $u$-equivalence nor $p$-equi\-va\-lence.
\end{remark}
%%%%%%%%%%%%%%%%%%%%%%%%%%%%%%%%%%%
%\subsection{Equivalence classes and invariance properties}

%The semi norms in the stability estimates in Theorems~\ref{thm:stability_general} or \ref{thm:stability_symmetric}
%reveal certain invariance properties and induce equivalence classes of
%forces in the following sense.

%%%%%%%%%%%%%%%%%%%%%%

In the remaining part of this section, we want to stress the following decomposition property for functionals induced by the mixed problem.
\begin{proposition}
The subspaces
\begin{align*}
  A(K) := \left \{ a(w^0, \bullet) \in V^\prime:
    w^0 \in K
    \right \} \quad \text{and} \quad
  K^0 = B^t(Q) := \left \{ b(\bullet, q) \in V^\prime:
    q \in Q
    \right \}
\end{align*}
are closed in $V^\prime$ and it holds
\begin{align}
    V^\prime = A(K) \oplus K^0.\label{eqn:decomposition_functionals}
\end{align}
\end{proposition}
\begin{proof}
The decomposition \eqref{eqn:decomposition_functionals}
follows directly from the unique solvability of the mixed problem and it remains to show that the two spaces are closed.
Since $B$ is surjective, the second statement follows
from the Banach closed range theorem
\cite[Theorem 4.1.5]{BBF}.
$A(K) \subset V^\prime$ is a subspace due to $a$ being bounded and bilinear.
The proof that $A(K)$ is also closed follows by elementary arguments using \eqref{eqn:a_coercivity_K}.

% Let us assume that $(a(w^0_k, \bullet))$ is a convergent sequence in $A(K)$ with limit $\bar{f} \in V^\prime$. Then, the sequence is a Cauchy sequence. For two elements
% $a(w^0_m, \bullet)$ and $a(w^0_n, \bullet)$
% of the sequence with $w^0_m \in K$ and $w^0_n \in K$,
% it holds, for all $v \in V$ with $\| v \|_V = 1$,
% $$
%   \left | a(w^0_m, v) - a(w^0_n, v) \right | = \left | a(w^0_m - w^0_n, v) \right |
%     \geq \alpha_0  \left | w^0_m - w^0_n \right |.
% $$
% Thus, $(w^0_k) \in K$ is also a Cauchy sequence and converges to some $\bar{w}^0$,
% since $K$ is closed. Further, it holds for all $k \in \mathbb{N}$ and all $v \in V$
% $$
%   \left | a(w^0_k, v) - a(\bar{w}^0, v) \right | =  \left | a(w^0_k - \bar{w}^0, v) \right |
%     \leq \| a \|  \| w^0_k - \bar{w}^0 \|_V  \| v \|_V,
% $$
% and $a(\bar{w}^0_k, \bullet) \to a(\bar{w}^0, \bullet)$ in $V^\prime$, with
% $a(\bar{w}^0, \bullet) \in A(K)$. Thus, it holds $\bar{f} = a(\bar{w}^0, \bullet) \in A(K)$,
% and therefore
% $A(K)$ is a closed subspace of $V^\prime$.
\end{proof}

\section{Examples}
\label{sec:examples}
This section discusses two finite dimensional and several infinite dimensional examples
that fit into the framework.

%%%%%%%%%%%%%%%%%%%%%%%%%%%%%%%%%%%%%%%%%%%%%%%%
\subsection{Finite dimensions ---  general case}
\label{sec:ex1}
In order to demonstrate the differences
between the classical and the
refined stability estimates, we study the finite-dimensional mixed problem
from \cite[Remark 3.4.6, p. 169]{BBF}. Euclidean norms are used for vectors, and the compatible spectral norms for matrices.
The two matrices
$$
  A := \begin{pmatrix}
        1 & - 1 \\
        1 & a
      \end{pmatrix}
\quad \text{and} \quad
  B := \begin{pmatrix}
         b & 0
      \end{pmatrix}
$$
are composed to the system matrix
$$
  M := \begin{pmatrix}
        A & B^T \\
        B & 0
  \end{pmatrix}
  = \begin{pmatrix}
      1 & -1 & b \\
      1 & a & 0 \\
      b & 0 & 0
  \end{pmatrix}.
$$
We are interested in the case
$0 < a \ll 1$ and $0 < b \ll 1$ to see the optimal dependence of the bounds on the parameters, but also in the
correct data dependence of the bounds.
The matrix $B$ induces the subspaces
\begin{align*}
  K & = \left \{ \lambda \begin{pmatrix} 0 \\ 1 \end{pmatrix} : \lambda \in \mathbb{R} \right \}
  \quad \text{and} \quad
K^0 = \left \{ \lambda \begin{pmatrix}  1 \\ 0 \end{pmatrix} : \lambda \in \mathbb{R} \right \},
\end{align*}
and the mixed problem induces a decomposition of the forcing
$\left (f_1, f_2 \right )^T$ as
\begin{align*}
  \mathbb{R}^2 & = A \cdot K \oplus K^0
   = \left \{ \lambda \begin{pmatrix} -1 \\ a \end{pmatrix} :
     \lambda \in \mathbb{R} \right \}
     \oplus K^0.
\end{align*}
Obviously, it holds $\beta = b$.
Further, due to $\ker B = \left \{ \lambda \left (0, 1 \right )^T : \lambda \in \mathbb{R} \right \}$ and
$$
  A_{KK} := \left (\ker B \right )^T
    \cdot A \cdot \ker B = \left ( a \right ),
$$
it holds $\alpha_0 = a$.
A singular value decomposition for $A$ shows that for $a\to 0$, the larger singular value converges
to $\sqrt\frac{2}{3-\sqrt{5}} = \| a \|$.

For the linear system
$$
  M \begin{pmatrix}
    u_1 \\ u_2 \\ p
  \end{pmatrix}
   = \begin{pmatrix}
      f_1 \\ f_2 \\ g
    \end{pmatrix},
$$
one gets the solution
\begin{equation} \label{example:finite:dim:1}
  u_1 = \frac{g}{b}, \quad
  u_2 = \frac{f_2}{a}-\frac{g}{b a}, \quad
  p=\frac{1}{b} \left (f_1 + \frac{f_2}{a} \right )
  - \frac{1}{b^2} \left (1 + \frac{1}{a}  \right ) g.
\end{equation}
The refined estimates
from Theorem \ref{thm:stability_general} read
\begin{align*}
  \Theta_{u, \text{r}} & := \frac{1}{\alpha_0}  \| f \|_{K^\prime} + \frac{1}{\beta} 
               \left ( 1 +\frac{\| a \|}{\alpha_0} \right ) \lvert g \rvert, \\
  \Theta_{p, \text{r}} & := \frac{1}{\beta} \inf_{w^0 \in K} \left( \| f - A w^0 \| +  \frac{\| a \|}{\alpha_0} \| f - A w^0 \|_{K^\prime} \right) + \frac{\| a \|}{\beta^2} \! \left ( 1 + \frac{\|a\|}{\alpha_0}\right ) \!
        \lvert g \rvert,
\end{align*}
where $f = (f_1,f_2)$ and $\| f \|_{K^\prime}
= |f_2|$, and are compared to the classical estimates
\begin{align*}
  \Theta_{u, \text{c}} & := \frac{1}{\alpha_0}  \| f \| + \frac{1}{\beta} 
               \left ( 1 +\frac{\| a \|}{\alpha_0} \right ) \lvert g \rvert, \\
  \Theta_{p, \text{c}} & := \frac{1}{\beta} \left(1 + \frac{\| a \|}{\alpha_0}\right) \| f \| + \frac{\| a \|}{\beta^2}  \left ( 1 + \frac{\|a\|}{\alpha_0}\right )
        \lvert g \rvert,
\end{align*}
which exploit full norms instead of semi norms.
Since $1 \leq \|a\|$, $\Theta_{u,r}$ perfectly reflects
the dependence of $u$ on $f$ and $g$ in \eqref{example:finite:dim:1}.

\begin{figure}
    \includegraphics[width = 0.49\textwidth]{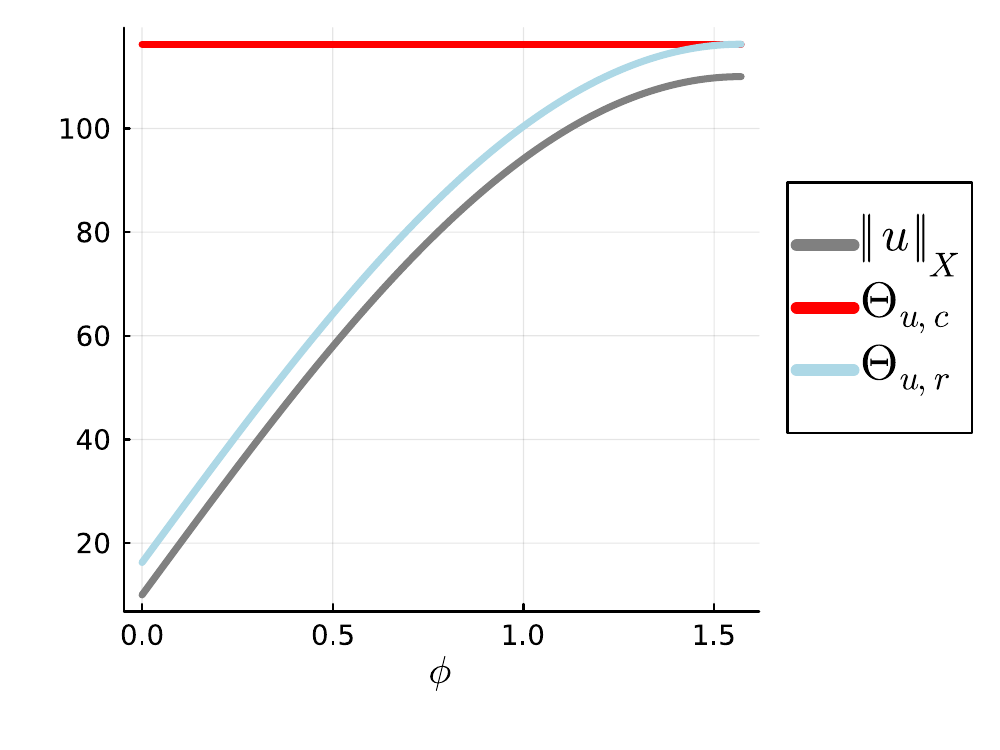}
    \hfill
    \includegraphics[width = 0.49\textwidth]{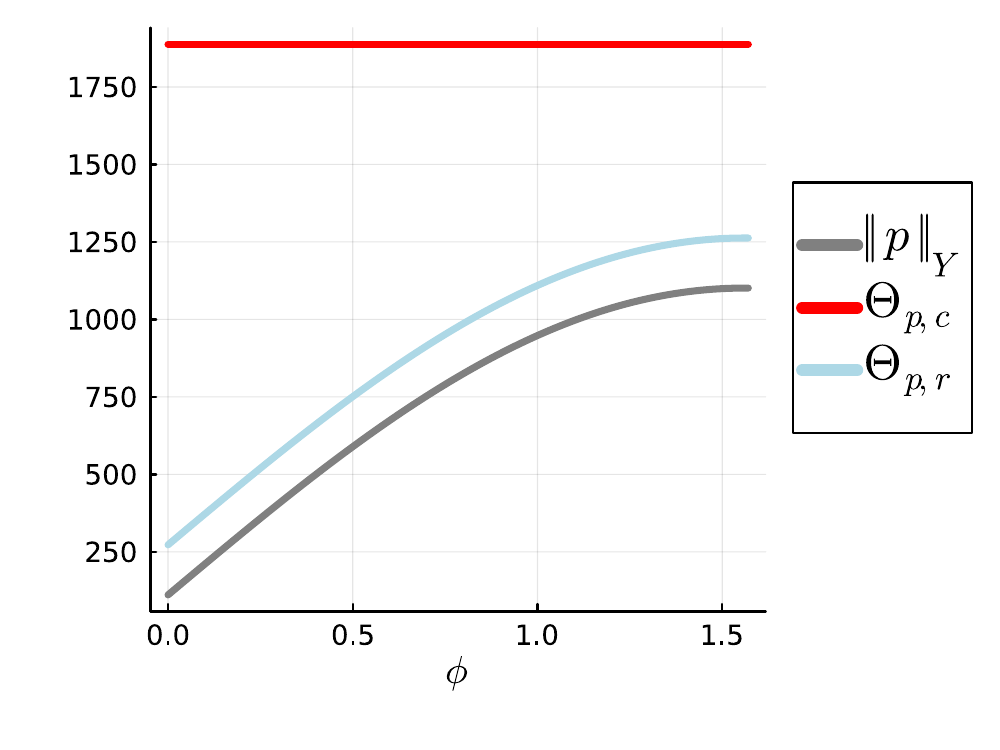}
    \caption{\label{fig:comparison_ex1}Exact norms, refined bounds and classical bounds in the example of Section~\ref{sec:ex1} with forces $(f_1,f_2) = (\cos(\phi), \sin(\phi))$ versus the $\phi$
    and parameters $a = 0.01$, $b = 0.1$ and $g = -0.01$.}
\end{figure}

Turning to the minimization process in the
estimate for $p$, observe that
only forces from $K^0$ change the variable $p$.
Hence, we seek $\lambda$ and $w^0 = \lambda (0,1)^t \in K$ such that
$$
  f-Aw^0 \in K^0 \quad \Leftrightarrow \quad
  \begin{pmatrix}
      f_1 \\
      f_2
  \end{pmatrix}
  - \lambda   \begin{pmatrix}
      -1 \\
      a
  \end{pmatrix}
  = \mu \begin{pmatrix}
      1 \\
      0
  \end{pmatrix}
  \quad \text{for some } \mu.
$$
This holds for $\lambda = f_2 / a$, resulting in
$$
  \begin{pmatrix}
      f_1 \\
      f_2
  \end{pmatrix}
  - \lambda   \begin{pmatrix}
      -1 \\
      a
  \end{pmatrix}
  = \begin{pmatrix}
      f_1 + f_2/a  \\
      0
  \end{pmatrix}.
$$
Thus, also $\Theta_{p,r}$ perfectly reflects
the dependence of $p$ on $f$ in \eqref{example:finite:dim:1}.

Figure~\ref{fig:comparison_ex1} shows the bounds and the exact norms for $u$ and $p$ with fixed pa\-ra\-me\-ters
$a = 0.01$, $b = 0.1$ and $g = -0.01$ and
forces $(f_1,f_2) = (\cos(\phi), \sin(\phi))$ versus $\phi$.
The refined bounds $\Theta_{u,r}$ and $\Theta_{p,r}$ appear closer to the real norms
than the classical bounds $\Theta_{u,c}$ and $\Theta_{p,c}$
for almost all $\phi$. Only
when $\phi$ is taken such that $\| f \|_{K^\prime} = \| f \|_{V^\prime}$, the classical bound for $u$ is as good as the refined bound.
Altogether, the example demonstrates the more predictive dependence of the refined bounds on $f_1$ and $f_2$.

%%%%%%%%%%%%%%%%%
\subsection{Finite dimensions --- symmetric case}
\label{sec:ex2}
The next example revisits \cite[Remark 3.4.7, p. 171]{BBF} and illustrates the results for the symmetric case.
Here, the mixed problem is defined by
$$
A := \begin{pmatrix}
  2 & \sqrt{a} \\
  \sqrt{a} & a
\end{pmatrix},
\quad \text{and} \quad
  B:= \begin{pmatrix}
    b & 0
  \end{pmatrix},
$$
leading to a system matrix
$$
  M = \begin{pmatrix}
        A & B^T \\
        B & 0
  \end{pmatrix}
  = \begin{pmatrix}
      2 & \sqrt{a} & b \\
      \sqrt{a} & a & 0 \\
      b & 0 & 0
  \end{pmatrix}
$$
with the solutions
\begin{equation} \label{example:finite:dim:2}
  u_1 = \frac{g}{b}, \quad
  u_2 = \frac{f_2}{a}-\frac{g}{b \sqrt{a}}, \quad
  p=\frac{1}{\sqrt{a}b} \left (\sqrt{a} f_1 - f_2 \right )
  - \frac{1}{b^2} g.
\end{equation}
Again, we are interested in the case $0 < a \ll 1$, $0 < b \ll 1$.
The space $K$ and its orthogonal complement read
\begin{align*}
K = \left \{ \lambda \begin{pmatrix} 0 \\ 1 \end{pmatrix} : \lambda \in \mathbb{R} \right \}
\quad \text{and} \quad
    K_a^\perp := \left\lbrace \lambda \begin{pmatrix}
        \sqrt{a}\\-1
    \end{pmatrix} : \lambda \in \mathbb{R}\right\rbrace.
\end{align*}
Moreover, one obtains the dual norms
\begin{align*}
\| f \|_{K^\prime} = \lvert f_2 \rvert, \quad
\| f \|_{(K_a^\perp)^\prime} = \frac{\lvert \sqrt{a} f_1 - f_2 \rvert}{\sqrt{a +1}}, \quad
\| g \|_{Q\prime} = \lvert g \rvert.
\end{align*}
A short calculation reveals
$\alpha \approx a/2$ and $\|a\| \leq 2$.
Hence, Theorem~\ref{thm:stability_symmetric} yields
  \begin{align}
    \|u\|_V & \leq \Theta_{u, \text{r}} := \frac{2}{a} \lvert f_2 \rvert + \frac{2}{\sqrt{a}b} \lvert g \rvert,\\
    \|p\|_Q & \leq \Theta_{p, \text{r}} := \frac{2}{\sqrt{a}b} 
    \frac{\lvert \sqrt{a} f_1 - f_2 \rvert}{\sqrt{a +1}}
    + \frac{2}{b^2} \lvert g \rvert
    = \frac{2}{\sqrt{a}b} 
    \frac{\left\lvert f_1 - \frac{f_2}{\sqrt{a}} \right\rvert}{\sqrt{\frac{a +1}{a}}}
    + \frac{2}{b^2} \lvert g \rvert.
  \end{align}
Up to constant factors this matches exactly the dependency on the data components $f_1$, $f_2$ and $g$
in the exact solution.

\begin{figure}
    \includegraphics[width = 0.49\textwidth]{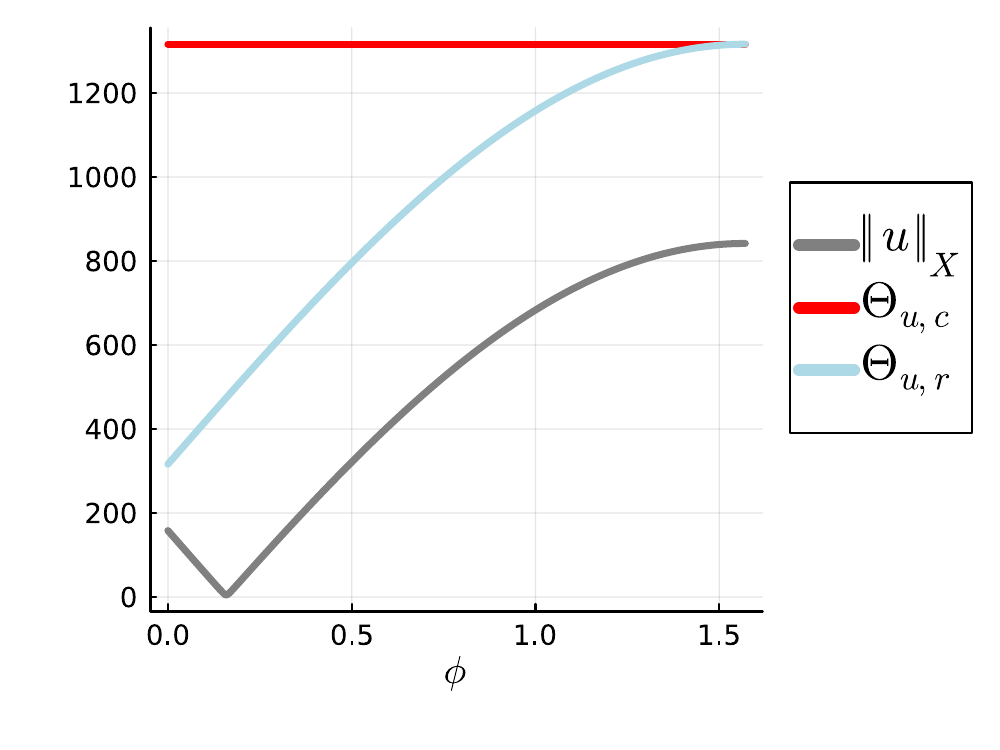}
    \hfill
    \includegraphics[width = 0.49\textwidth]{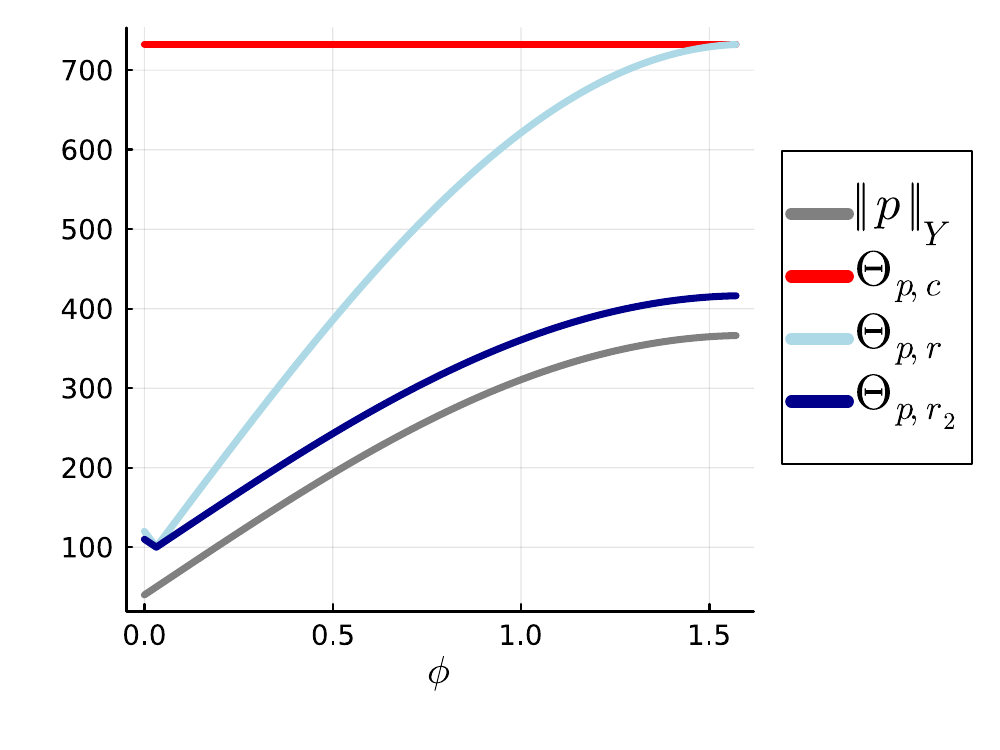}
    \caption{\label{fig:comparison_ex2}Exact norms, refined bounds and classical bounds in the example of Section~\ref{sec:ex2} with forces $(f_1,f_2) = (\cos(\phi), \sin(\phi))$ versus the $\phi$
    and parameters $a = 0.001$, $b = 0.1$ and $g = 0.5$.}
\end{figure}

Figure~\ref{fig:comparison_ex2} compares the refined bounds to 
classical bounds of the form
  \begin{align}
    \Theta_{u, \text{c}} := \frac{2}{a} \| f \| + \frac{2}{\sqrt{a}b} \lvert g \rvert
    \quad \text{and} \quad
    \Theta_{p, \text{c}} := \frac{2}{\sqrt{a}b} 
    \| f \|
    + \frac{2}{b^2} \lvert g \rvert.
  \end{align}
in a similar experiment as in the first example.
The forcing scans through all possible vectors via the
choice $f = (f_1,f_2) = (\cos(\phi), \sin(\phi))$ for
$\phi \in [0,\pi/2]$
and the other parameters are fixed by $a = 0.001$, $b = 0.1$ and $g = 0.5$. Additionally also the sharper refined
bound that involves $\| f - A u_f \|$ with $u_f = (0, f_2/a)$ is shown and reads
\begin{align*}
    \Theta_{p, \text{r}_2} := \frac{2}{\sqrt{a}b} 
    \left\lvert f_1 - \frac{f_2}{\sqrt{\alpha}} \right\rvert
    + \frac{2}{b^2} \lvert g \rvert.
\end{align*}
The conclusions are similar to the first example: through
the employment of semi norms much sharper a priori error
bounds are possible with dramatic differences in cases
where the forcing $f$ has a large part in $K^0$, i.e.\
for small $\phi$.

%%%%%%%%%%%%%%%%%%%%%%%%%%%%%%%%%%%%%%%%%%%
\subsection{Helmholtz--Hodge decomposition} \label{lab:class:HHD}
A typical application of mixed problems
are decompositions. As a first
application, we will briefly describe
the decomposition of an $\bL^2$ vector field
into the sum of a gradient and a divergence-free field
on a Lipschitz domain $\Omega$.
For this Helmholtz--Hodge decomposition
we choose as Hilbert spaces
$V:=\bL^2$ and $Q:=H^1 \big / \mathbb{R}$, equipped
with the norms $\| \bu \|_V := (\bu, \bu)_{L^2}^{1/2}$ and 
$\| p \|_Q := (\nabla p, \nabla p)_{L^2}^{1/2}$, respectively.
The bilinear forms read
\begin{align*}
    a(\bu, \bv) := \left ( \bu, \bv \right )_{L^2}
    \quad \text{and} \quad
    b(\bu, q) := \left ( \bu, \nabla q \right )_{L^2}.
\end{align*}
Thus, by
\begin{align*}
  \hspace{2cm}
  a(\bu, \bv) + b(\bv, p) & = (\bfe, \bv)_{L^2}
  && \text{for all } \bv \in V,
  \hspace{2cm}\\
   b(\bu, q) & = 0
    && \text{for all } q \in Q,
\end{align*}
an arbitrary vector field $\bfe \in \bL^2$
is decomposed
into $\bfe = \bu + \nabla p$, where
$\bu \in \bL^2_\sigma := \{ \bv \in \bL^2: (\bv, \nabla q)_{L^2} = 0 \, \text{ for all $q \in H^1 \big / \mathbb{R} $} \}$, i.e.,
the kernel $K$ is given by the space
of solenoidal (divergence-free) vector fields, whose normal flux vanishes at the boundary of the domain.
The well-posedness of the Helmholtz--Hodge
decomposition is shown by the above framework.
Indeed, $a$ is symmetric and coercive with $\| a \|=1$
and $\alpha=1$, and $b$ satisfies
$\| b \| = 1$ and $\beta=1$. Further,
we have $g = 0$.

Due to $b(\bu, q) = 0$, the
decomposition is orthogonal in the $\bL^2$
standard scalar product and one gets
$K_a^\perp = \{ \nabla q : q \in H^1 / \mathbb{R} \}$.
Theorem \ref{thm:stability_symmetric}
yields the estimates
\begin{align*}
  \| \bu \|_{\bL^2} \leq \|  \bfe \|_{\left (\bL^2_\sigma \right )'}
  \quad \text{and} \quad
  \| \nabla p \|_{\bL^2} \leq \|  \bfe \|_{\left (\nabla H^1\right )'}.
\end{align*}
Two \emph{consistency checks} show \emph{optimality}
w.r.t.\ the data $\bfe$ in the sense that
\begin{align*}
i) \ & \bfe = \bu \in \bL^2_\sigma \text{ delivers }
  \| \bu \|_{\bL^2} \leq \|  \bu \|_{\left (\bL^2_\sigma \right )'}
  \text{ and }
  \| \nabla p \|_{\bL^2} \leq 0,\\
ii) \ & \bfe = \nabla p \text{ with } p \in Q \text{ delivers }
  \| \bu \|_{\bL^2} \leq 0
  \text{ and }
  \| \nabla p \|_{\bL^2} \leq \| \nabla p \  \|_{\left (\nabla H^1\right )'}.
\end{align*}
\emph{Optimality} is derived by \emph{semi norm arguments}, and
the inequalities
$\| \bu \|_{\bL^2} \leq \|  \bu \|_{\left (\bL^2_\sigma \right )'}$
and $\| \nabla p \|_{\bL^2} \leq \| \nabla p \  \|_{\left (\nabla H^1\right )'}$  are actually \emph{equalities}.

%%%%%%%%%%%%%%%%%%%%%%%%%%%%%%%%%%%%%%%%%%%%%%%%%%%
\subsection{Another Helmholtz--Hodge decomposition}
There is another decomposition, on Lipschitz domains $\Omega$, of an $\bL^2$ vector field
into the sum of a curl-free function and a divergence-free field. To simplify, we only discuss the two-dimensional case. 
The operators $\mathbf{curl}: Q \rightarrow V$
and $\mathrm{curl}: V \rightarrow Q$ are defined by
\begin{align*}
  \mathbf{curl} \, q := \begin{pmatrix}
      \frac{\partial q}{\partial y} \\
      -\frac{\partial q}{\partial x} \\
  \end{pmatrix}
  \quad \text{and} \quad
  \mathrm{curl} \, \mathbf{v} := \frac{\partial v_1}{\partial y} - \frac{\partial v_2}{\partial x}.
\end{align*}
The Hilbert spaces
$V:=\bL^2(\Omega)$ and $Q:=H^1 / \mathbb{R}$ are equipped
with the norms $\| \bu \|_V := (\bu, \bu)_{L^2}^{1/2}$ and 
$\| p \|_Q := (\nabla p, \nabla p)_{L^2}^{1/2}
= (\mathbf{curl}\, p, \mathbf{curl}\, p)_{L^2}^{1/2}$.
The bilinear forms read
\begin{align*}
    a(\bu, \bv) := \left ( \bu, \bv \right )_{L^2}
    \quad \text{and} \quad
    b(\bv, q) := \left ( \bv, \mathbf{curl}\,  q \right )_{L^2}.
\end{align*}
Here, $a$ is symmetric and coercive with $\| a \|=1$
and $\alpha=1$, and for $b$ it holds
$\| b \| = 1$ and $\beta=1$.
The associated mixed problem decomposes a functional $\bfe \in V$ into
\begin{align*}
  \bfe = \bu + \mathbf{curl}\,  p
\end{align*}
with $p \in H^1(\Omega) / \mathbb{R}$ and $\bu \in K := \left\lbrace L^2(\Omega) : (\bu, \mathbf{curl}\,  q)_{L^2} = 0 \text{ for all } q \in Q \right\rbrace$, i.e., the space of curl-free vector fields with vanishing tangential flux at the domain boundary. Consequently, it holds
$K_a^\perp := \lbrace \mathbf{curl}\,  q : q \in Q \rbrace$.
Theorem \ref{thm:stability_symmetric} yields the estimates
\begin{align*}
  \| \bu \|_{\bL^2} \leq \|  \bfe \|_{K'}
  \quad \text{and} \quad
  \| \mathbf{curl}\,  p \|_{\bL^2} \leq \|  \bfe \|_{\left (\mathbf{curl}\,  H^1\right )'},
\end{align*}
which obey similar \emph{optimality} properties
w.r.t.\ the data $\bfe$ like
in Subsection \ref{lab:class:HHD}.

%%%%%%%%%%%%%%%%%%%%%%%%%%%%%%%%%%%%%%%%%%%%%%%%%%%%%%%%%%%%%%%%%%%%
\subsection{The Stokes problem with homogeneous Dirichlet boundary conditions}
Consider the Stokes problem with homogeneous Dirichlet
velocity (no-slip) boundary conditions and a parameter $\mu > 0$. Its strong form seeks $\bu$ and $p$ with
\begin{align}\label{eqn:stokes_problem}
  -\mu \Delta \bu + \nabla p = \bfe
  \quad \text{and} \quad
  \mathrm{div} \, \bu = 0.
\end{align}
The problem is
represented as a mixed problem by
\begin{align*}
  a(\bu, \bv) + b(\bv, p) & = \langle \bfe, \bv \rangle
  && \text{for all } \bv \in V, \\
   b(\bu, q) & = 0
   && \text{for all } q \in Q,
\end{align*}
with $V := \bH^1_0$ and $Q := L^2 \big / \bbR$
and the bilinear forms
\begin{align*}
    a(\bu,\bv) := \mu (\nabla \bu, \nabla \bv)_{L^2}
    \quad \text{and} \quad
    b(\bv,q) := -(\mathrm{div} \, \bv, q)_{L^2}.
\end{align*}
Thus, the Stokes problem
induces the \emph{Stokes decomposition} of $V^\prime$ into
\begin{equation} \label{eq:Stokes:decomposition}
  \left ( \bH^1_0 \right )' = \left \{a(\bu_0, \bullet) : \bu_0 \in K \right \} \oplus \left \{b(\bullet, q_0) : q_0 \in Q \right \}
  \quad \text{where} \quad K := \bL^2_\sigma \cap \bH^1_0.
\end{equation}
Clearly,
the decomposition of the linear functional is independent of $\mu$.

Similar to the previous example, Theorem~\ref{thm:stability_symmetric}
yields the bounds
\begin{align*}
  \| \nabla \bu \|_{\bL^2} \leq \mu^{-1} \|  \bfe \|_{K'}
  \quad \text{and} \quad
  \| p \|_{L^2} \leq \beta^{-1} \|  \bfe \|_{(K_a^\perp)'},
\end{align*}
where $\beta$ is the inf-sup constant of the Stokes problem.
The \emph{consistency checks} from above for the Helmholtz--Hodge decomposition read for the Stokes problem like: i) Setting
$\bfe = a(\bu, \bullet) =: \langle -\mu \Delta \bu, \bullet \rangle_{V^\prime \times V} $ for $\bu \in K$, yields
\begin{align*}
  \| \nabla \bu \|_{\bL^2} \leq \mu^{-1} \|  -\mu \Delta \bu  \|_{K'}
    = \| \Delta \bu  \|_{K'}
    \quad \text{and} \quad
  \| p \|_{L^2} \leq 0,
\end{align*}
and ii) setting $\bfe = b(\bullet, p) =: \langle \nabla p, \bullet \rangle_{V^\prime \times V} $ for $p \in Q$
\begin{align*}
  \| \nabla \bu \|_{\bL^2} \leq 0
  \quad \text{and} \quad
  \| p \|_{L^2} \leq \beta^{-1} \|  \nabla p \|_{(K_a^\perp)'}.
\end{align*}
Thus, the Stokes problem above has two distinct physical regimes, induced by
the decomposition of $V'$.
The estimate in i) shows a dynamic behavior induced
by forcings $\bfe = a(\bu_0, \bullet)$ for any $\bu_0 \in K$.
Observe, that the parameter $\mu$ drops out from the stability estimate and does not cause any \emph{stability loss}. This is important
for \emph{pressure-robust} Stokes discretizations, see Section~\ref{sec:pressure:robustness} below.
The estimate in ii) manifests the physical
regime of \emph{hydrostatics}, induced by gradient forcings
$\nabla \phi \in \bH^{-1}$ with $\phi \in Q$.

Like in Subsection \ref{lab:class:HHD}, the inequality
$\| \nabla \bu \|_{\bL^2} \leq \| \Delta \bu  \|_{K'}$ is actually
an equality.
The inequality
$\| p \|_{\bL^2}  \leq \beta^{-1} \|  \nabla p \|_{(K_a^\perp)'}$
mirrors the classical norm equivalence
$$
  \beta \| q \|_{L^2} \leq \| \nabla q \|_{(\bH^1_0)'}
  \leq \| q \|_{L^2}
  \quad \text{for all } q \in Q.
$$

%%%%%%%%%%%%%%%%%%%%%%%%%%%%%%%%%%%%%%%%
\subsection{The Stokes problem with inhomogeneous Dirichlet data}
The role of semi norms for establishing \emph{physical regimes} requires
to handle more general boundary conditions. When dealing with inhomogeneous velocity Dirichlet (slip) boundary conditions $\bu = \bu_0$ along $\partial \Omega$
for the Stokes problem \eqref{eqn:stokes_problem}
not much changes, because 
the weak formulation still
employs \emph{velocity test functions} $\bv \in \bH^1_0$ with zero boundary
conditions. To abstain from technicalities, we just argue that the following two observations hold for solutions $(\bu_j,p_j)$ for data $\bfe_j, j \in \lbrace1,2 \rbrace$:
\begin{align*}
  \bfe_1 - \bfe_2 = \nabla p_0 \text{ for } p_0 \in Q & \quad \Rightarrow \quad \bu_1 = \bu_2 \ \left ( \wedge \ p_1 - p_2 = p_0 \right ), \quad (\text{$\bu$-equivalence}) \\
  \bfe_1 - \bfe_2 = -\mu \Delta \bu_0 \text{ for } \bu_0 \in V & \quad \Rightarrow \quad p_1 = p_2 \ \left ( \wedge \ \bu_1 - \bu_2 = \bu_0 \right ), \quad (\text{$p$-equivalence}),
\end{align*}
with $V:=\bH^1_0$ and $Q:=L^2 \big / \mathbb{R}$ like in the Stokes problem
with \emph{homogeneous} Dirichlet velocity boundary conditions, since
$\bu_1 - \bu_2 = \bu_0 \in V$ has zero boundary conditions.

%%%%%%%%%%%%%%%%%%%%%%%%%%%%%%%%%%%%%%%%%%%%%%%%%%%%%%%%%%%%
\section{Extending models: from the Stokes problem to
the transient (nonlinear) Navier--Stokes problem} \label{sec:Inhom:NSE}
A major question in the (space) discretization theory for partial
differential equations is, which relevance numerical analysis
for linear model problems (like the Stokes problem)
may have for more complicated 
models like the transient incompressible Navier--Stokes equations.
This section explains how \emph{semi norm} arguments can help to
bridge the gap between these worlds. 
%The transient incompressible
%Navier--Stokes equations and the investigation of certain \emph{physical regimes} therein will serve as a model.

%%%%%%%%%%%%%%%%%%%%%%%%%%%%%%%%%%%%%%%%%%%%%%%%%%%%%%%%%%%%%%%%%%%%%%%%
%\subsubsection{Extending the Stokes problem to the steady and transient Navier--Stokes equations}
\subsection{The Navier--Stokes equations}
Assuming inhomogeneous Dirichlet velocity boundary conditions and $\mu > 0$ and the Stokes problem
\begin{align*}
  -\mu \Delta \bu + \nabla p = \bfe
  \quad \text{and} \quad
  \mathrm{div} \, \bu = 0,
\end{align*}
adding the term $\left (\bu \cdot \nabla \right ) \bu$ leads to the steady incompressible Navier--Stokes equations
\begin{align*}
  -\mu \Delta \bu + \left (\bu \cdot \nabla \right ) \bu + \nabla p = \bfe
  \quad \text{and} \quad
  \mathrm{div} \, \bu = 0.
\end{align*}

Obviously, the Navier--Stokes problem is no longer a \emph{mixed problem} in the sense of \eqref{eqn:model_problem},
since it is nonlinear and also does not have a unique solution in general. Nevertheless, the \emph{$u$-equivalence} property of the Stokes problem is preserved for the steady incompressible
Navier--Stokes equations, since it depends only on the interplay
of the weak formulations $\bv \in \bH^1_0 \to b(\bv, p)$ (for $\nabla p$)
and $q \in L^2 \big / \mathbb{R} \to b(\bu, q)$ (for $\mathrm{div} \, \bu$).
Thus, for all $\phi \in L^2 \big / \mathbb{R}$, the forcing
$\bfe = \nabla \phi$ allows for a \emph{hydrostatic physical regime}
with solution $(\bu, p) = (\bzero, \phi)$. The
additional term $\left (\bu \cdot \nabla \right ) \bu$
vanishes for $\bu=\bzero$. Following the explanations in Section~\ref{sec:Inhom:NSE}, it still holds
\begin{align*}
  p_0 \in Q \ \wedge \ \bfe_1 - \bfe_2 = \nabla p_0 
  & \quad \Rightarrow \quad
  \bu_1 = \bu_2 \ \left ( \wedge \ p_1 - p_2 = p_ 0 \right ).
\end{align*}
This reflects the fact that also the data in the stability estimate for the
steady incompressible Navier--Stokes equations is measured in
the semi norm $\| \bullet \|_{(K)'}$.
Indeed, like in the linear Stokes problem, the space $K$ represents
the divergence constraint for the Navier--Stokes velocity solution $\bu$, and
the polar space $K^0 \subset (\bH^1_0)^\prime$ represents the gradient fields that
do no influence the dynamics of the Navier--Stokes equations, but only
the pressure field $p$.
More importantly, also
the nonlinear term
$\left ( \bu \cdot \nabla \right ) \bu$ can be investigated in light
of the semi norm $\| \bullet \|_{(K)'}$, since the Stokes decomposition
\eqref{eq:Stokes:decomposition} can be applied to any functional $\bfe \in V'$ and in particular to
$\left ( \bu \cdot \nabla \right ) \bu \in V'$.

%%%%%%%%%%%%%%%%%%%%%%%%%%%%%%%%%%%%
\subsection{Physical regimes with irrotational convection terms}
Especially interesting are velocities $\bu$, for which
it happens to hold
$\left ( \bu \cdot \nabla \right ) \bu = \nabla q_0 \in K^0$ with
$q_0 \in Q$.
In this situation, the velocity solution $\bu$ of the linear
Stokes problem is also a velocity solution of the nonlinear Navier--Stokes
equations, and only the pressure changes to
$p_{\mathrm{nonlin}} = p_{\mathrm{lin}} - p_0$.
Actually, this set of velocity solutions $\bu$ builds an important
\emph{physical regime} of the nonlinear Navier--Stokes problem characterized by
$$
  \| \left ( \bu \cdot \nabla \right ) \bu \|_{K'} = 0,
$$
i.e., the strength of the nonlinear convection term (in this certain semi norm) is zero.
Indeed, such flows are called \emph{generalized Beltrami flows} and
consist of at least three subregimes: \emph{channel flows},
\emph{potential flows} and
\emph{vortical flows}.
We remark that all these flows are only driven by boundary conditions, i.e., it holds $\bfe = \bzero$. The three subregimes are discussed in the next subsections.

\subsubsection{Channel flows}
For channel flows like the classical Hagen-Poiseuille flow
it holds
$\left ( \bu \cdot \nabla \right ) \bu = \bzero$.
Thus, channel flows are a very special physical
regime, where it is not important that the nonlinear convection
term is measured in a semi norm or a norm, since it is identically zero,
anyway.
In consequence, the Stokes and
the Navier--Stokes solution also coincide with respect to the pressure $p$.

\subsubsection{Potential flows}
\label{sec:potential_flows}
For a potential flow, a harmonic function $h$ is given, i.e.,
$-\Delta h = 0$
and it holds $\bu = \nabla h$. Thus, the divergence constraint is fulfilled
due to
$$
  \mathrm{div} \, \bu = \mathrm{div} \, \nabla h = \Delta h = 0.
$$
Regarding the momentum equations, observe that
$-\mu \Delta \bu = -\mu \Delta \nabla h = -\mu \nabla \Delta h = \bzero$.
Thus, $\bu$ fulfills the incompressible Stokes equations
together with the pressure $p=0$, since
the Stokes momentum balance reads
$$
  \bzero= - \mu \Delta \bu + \nabla p = \nabla p.
$$
We now show that $\bu$ also fulfills the incompressible Navier--Stokes
equations, but turning from Stokes to Navier--Stokes the pressure
$p$ has to be adapted appropriately.
Therefore, we remind the reader that it holds
pointwise
$$
  \left ( \bu \cdot \nabla \right ) \bu = \bomega \times \bu
   + \frac{1}{2} \nabla \left ( \bu^2 \right )
\quad \text{with the \emph{Lamb vector}} \quad
  \bomega := \nabla \times \bu.
$$
For potential flows, $\bomega$ vanishes due to
\(\bomega = \nabla \times \bu = \nabla \times \nabla h = \bzero\)
and hence 
$$
  \left ( \bu \cdot \nabla \right ) \bu = \frac{1}{2} \nabla \left ( \bu^2 \right )
$$
is a gradient field.
Here, the semi norm argument is decisive: in potential flows
the nonlinear convection term is not zero in the norm
$\| \bullet \|_{V'}$, but it is zero in the norm
$\| \bullet \|_{K'}$.
Therefore, the Navier--Stokes momentum balance
boils down to
$$
  \bzero = -\mu \Delta \bu + \left ( \bu \cdot \nabla \right ) \bu
  + \nabla p  = \frac{1}{2} \nabla \left ( \bu^2 \right ) + \nabla p,
$$
and the Navier--Stokes pressure $p$ is given by $p=-\frac{1}{2} \bu^2 + C$,
where $C$ is a domain-dependent constant such that it holds $p \in Q$.

%%%%%%%%%%%%%%%%%%%%%%%%%%%%%%%%%%%%%%%
\subsection{Generalized Beltrami flows}
A third physical regime, where
the flow $\bu$ has neither a vanishing nonlinear convection term, nor satisfies $-\mu \Delta \bu = \bzero$ are the so-called \emph{(generalized) Beltrami flows}. Here,
the Lamb vector is either zero (Beltrami flows) or
a gradient field $\nabla \phi_0$ (generalized Beltrami flows), such that
$$
  \left ( \bu \cdot \nabla \right ) \bu
  = \bomega \times \bu
   + \frac{1}{2} \nabla \left ( \bu^2 \right )
   = \nabla \phi_0 + \frac{1}{2} \nabla \left ( \bu^2 \right ).
$$
is balanced with a pressure $p=-\phi_0 -\frac{1}{2} \bu^2 + C$.

Actually, generalized Beltrami flows
deliver vortical flows like the Taylor--Green vortex.
This can be explained from a physical point of view:
For a (steady) radially symmetric vortex, the nonlinear convection term can be understood as the (local) centrifugal force
\cite[p. 39f.]{MR3363553}. As a gradient field it is balanced by the pressure gradient
(which enforces the divergence constraint) and the vortex
does not explode.

Finally, we want to emphasize that all steady Navier--Stokes equations
can be transformed to a transient one by Galilean invariance.
For a constant velocity field $\bw_0$ and a steady Navier--Stokes solution
$(\bu_0, p_0)$ the transformations
$$
    \bu(t, \bx) := \bw_0 + \bu_0(t, \bx - t \bw_0),
    \qquad
    p(t, \bx) := p_0(\bx - t \bw_0)
$$
are transient Navier--Stokes solutions.
Since vortices play an important role in the solution theory of the
incompressible Navier--Stokes problem at hogh Reynolds numbers, it is interesting to connect them to solutions of the steady Stokes problem, where
the nonlinear convection term is zero, measured in the appropriate semi norm
above.

%%%%%%%%%%%%%%%%%%%%%%%%%%%%%%%%%%%%%
\subsection{Further physical regimes and extensions}
Adding further terms from more general physical models
leads to additional physical regimes. Adding, e.g., the Coriolis force
to the Stokes or Navier--Stokes momentum balance due to a rotating frame (with constant angular velocity $\boldsymbol{\Omega}$) allows 
\emph{geostrophic flows} characterized by $\| 2 \boldsymbol{\Omega} \times \bu\|_{K^\prime} = 0$
measured in the appropriate semi norm.

Moreover, this observation is not restricted to the incompressible
Navier--Stokes equations. Similar observations have been made
for (nearly) incompressible thermo-elasticity \cite{MR4213012},
where the right hand side is (nearly) zero, measured in an appropriate
semi norm. Further generalizations can be made by looking at singular
perturbation theory, which allows, e.g., to extend
considerations on the role of physical regimes induced by
semi norms like vortical flows, geostrophical flows, \ldots, from incompressible flows to \emph{low Mach number compressible flows} \cite{AKBAS2020113069}.

%%%%%%%%%%%%%%%%%%%%%%%%%%%%%%%%%%%%%%%%%%%%%%%%%%%
\section{Pressure-robustness in the Stokes problem}\label{sec:pressure:robustness}
The importance of semi norms for mixed problems was
realized for the first time
in the context of \emph{pressure-robust}
finite element discretizations for the incompressible Stokes
and Navier--Stokes
equations \cite{MR3564690}, i.e., in finite dimensional
(spatial) discretizations of an infinite dimensional problem.
The notion \emph{pressure-robustness} emphasizes
that there exists a subclass of mixed finite element discretizations
preserving $u$-equivalence under spatial discretization.
Referring to the Stokes decomposition \eqref{eq:Stokes:decomposition}, this means that
the (discrete) dynamics of pressure-robust mixed finite elements
is not affected by any forcing in $K^0$, i.e., by arbitrary gradient
fields in the sense of $(\bH^1_0)'$. In other words, pressure-robust
mixed finite elements ensure that any gradient field from $K^0$ is balanced only by discrete pressure gradients. We illustrate two consequences
by examples:
\begin{itemize}
\item[(i)] The discrete stability estimates
of pressure-robust mixed finite element methods resemble
the stability estimate of the corresponding infinite dimensional
problem better than non-pressure-robust discretizations;
\item[(ii)] Non-pressure-robust space discretizations suffer from a certain
(possibly huge) \emph{consistency error}, while pressure-robust discretizations do not.
\end{itemize}
%%%%%%%%%%%%%%%%%%%%%%%%%%%%%%%

%The refined estimates and previous Sections discuss
%the importance of semi norms to identify physical regimes.
%The question is if this has any implications for
%numerics beyond the mere stability of the scheme
%and unique existence of discrete solutions.
%More precisely, one may ask if there are discretizations that are able to preserve the $u$ or $p$-equivalence,
%such that also the discrete solution does not change
%if the associated parts of the data are unchanged.
%
%\medskip
%This section presents a simple numerical example
%that demonstrates that there are finite element discretizations
%that can mimic the $u$-equivalence property induced by the
%semi norm in the refined bounds, so called pressure-robust finite element %methods. For such methods the refined
%stability bounds allow for a much sharper numerical
%analysis and the identification of structure-preserving properties.
%
%\medskip

%%%%%%%%%%%%%%%%%%%%%%%%%%%%%%%%%%%%%%%%%%%%%%
\subsection{Pressure-robustness and discrete stability estimates}
The first example considers the right-hand side
\begin{align*}
  \bfe := \lambda \bfe_u + (1-\lambda) \bfe_p
  & := -\lambda \mu \Delta(\mathbf{curl}(x^2(x-1)^2y^2(y-1)^2)) + (1-\lambda) \nabla (x^3)
\end{align*}
with some coefficient $\lambda \in [0,1]$.
Seen as a functional $\bfe \in (\bH^1_0)'$ we have
\begin{align*}
  \bfe = a(\bu(\lambda), \bullet) + b(\bullet, p(\lambda))
\end{align*}
with the exact solutions (depending on $\lambda$) given by
\begin{align*}
    \bu(\lambda) = \lambda \, \mathbf{curl}(x^2(x-1)^2y^2(y-1)^2))    
\quad \text{and} \quad
p(\lambda) = (1-\lambda)x^3.
\end{align*}
Hence, for $\lambda = 1$ the pressure $p$ vanishes and for $\lambda = 0$ the velocity $\bu$ vanishes.
Observe that the operator norms in Theorem~\ref{thm:stability_symmetric} can be 
calculated by
\begin{align*}
  \| \bfe \|_{K^\prime}
  & = \| a(\bu,\bullet) \|_{K^\prime}
  = \mu \| \nabla \bu \|_{L^2},\\
  \| \bfe \|_{(K_a^\perp)^\prime}
  & = \| b(\bullet, p) \|_{(K_a^\perp)^\prime}
  = \| (\mathrm{div} \bullet, p) \|_{(K_a^\perp)^\prime}
  \leq \| p \|_{L^2}.
\end{align*}

Figure~\ref{fig:example_stability_stokes} compares
the norms of two discrete solutions $(\bu_\text{TH}, p_\text{TH})$
and $(\bu_\text{SV}, p_\text{SV})$ computed by the second-order Taylor--Hood finite element method \cite{taylor1973numerical} and 
the di\-ver\-gen\-ce-free second-order Scott--Vogelius \cite{SV1983} finite element method (on a barycentrically refined mesh), respectively. The left part of the figure also depicts two upper bounds
for the velocity norm, namely
\begin{align*}
    \Theta_{u,c} & := \mu^{-1} \| \bfe \|_{V^\prime} \leq \lambda \| \nabla \bu \|_{L^2} + (1-\lambda) \| p \|_{L^2},\\
    \Theta_{u,r} & := \mu^{-1} \| \bfe \|_{K^\prime} = \lambda \| \nabla \bu \|_{L^2}.
\end{align*}
The bound $\Theta_\text{u,c}$ is computed according to
\cite[Theorem 4.2.3]{BBF}, while the bound $\Theta_\text{u,r}$ is computed according to Theorem~\ref{thm:stability_symmetric}.
Similarly, the right part of the figure depicts 
the respective upper bounds for the pressure, namely
\begin{align*}
    \Theta_{p,c} := \beta^{-1} \| \bfe \|_{V^\prime}
    \quad \text{and} \quad
    \Theta_{p,r} := \beta^{-1} \| \bfe \|_{(K_a^\perp)^\prime}
    \leq \frac{1}{\beta} \| p \|_{L^2}.
\end{align*}
The inf-sup constant $\beta = 0.382683$
is a valid lower bound on the unit square \cite{Stoyan}.

\begin{remark}[$p$-equivalence]
Both methods are not $p$-equivalent and $\Theta_{p,r}$ is not a valid bound for $\|p_\text{TH}\|_{L^2}$ or $\|p_\text{SV}\|_{L^2}$ (or only with accounting for $\| a(\bu,\bullet) \|_{(K_{a,h}^\perp)^\prime}$). Due to the small parameter $\mu$ in the example, the error is negligible.
\end{remark}

Both selected discretizations are conforming in the sense that $V_h \subset V$ and $Q_h \subset Q$, and inf-sup stable, leading to
convergence with optimal order. The discrete stability estimates involve the kernel $K_h$, which differs depending on the method: the Scott--Vogelius finite element method is divergence-free, meaning $Q_h = \mathrm{div}(V_h)$, and therefore $K_h$ is a subset of $K$. Thus, the kernel of the semi norm $\| \bullet \|_{(K_h)'}$
is a superset of the kernel of the semi norm $\| \bullet \|_{(K)'}$,
and the discrete Scott--Vogelius resembles the
infinite dimensional stability estimate for the Stokes problem.
On the other hand, the Taylor--Hood finite element method is not divergence-free and it holds $K_h \not\subset K$. In this case,
functionals in the kernel of the semi norm
$\| \bullet \|_{(K)'}$ need not to be in
the kernel  of the discrete semi norm $\| \bullet \|_{(K_h)'}$.
Thus,
$\Theta_\text{u,r}$ above is in general not an upper bound
for $\bu_\text{TH}$.
Since $\| \bfe \|_{K_h^\prime} \leq \| \bfe \|_{V^\prime}$, the classical bound $\Theta_{u, c}$ is still valid. 

%%%%%%%%%%%%%%%%%%%%%%%%%%%%%%%%%%%%%%%%%%%%%%

\begin{figure}
\includegraphics[width=0.325\textwidth]{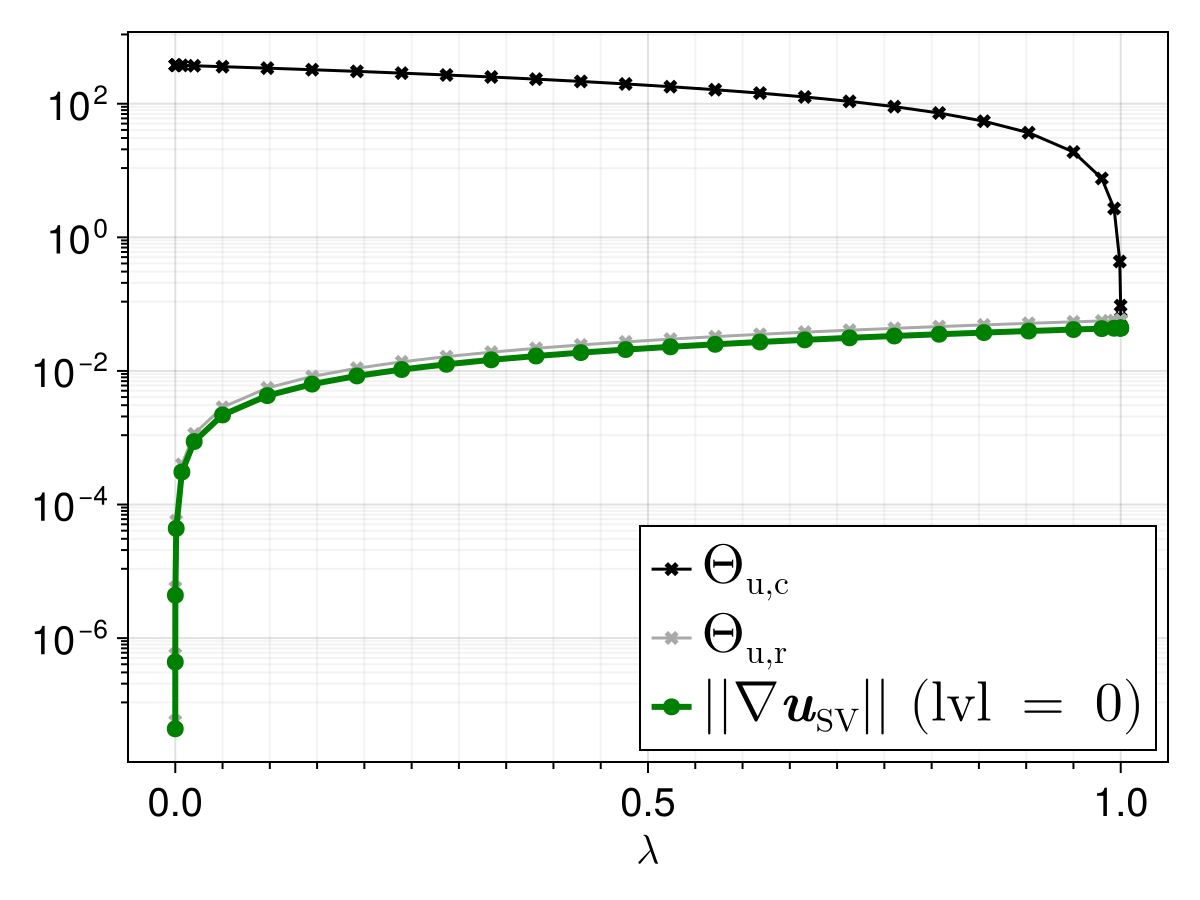}
\includegraphics[width=0.325\textwidth]{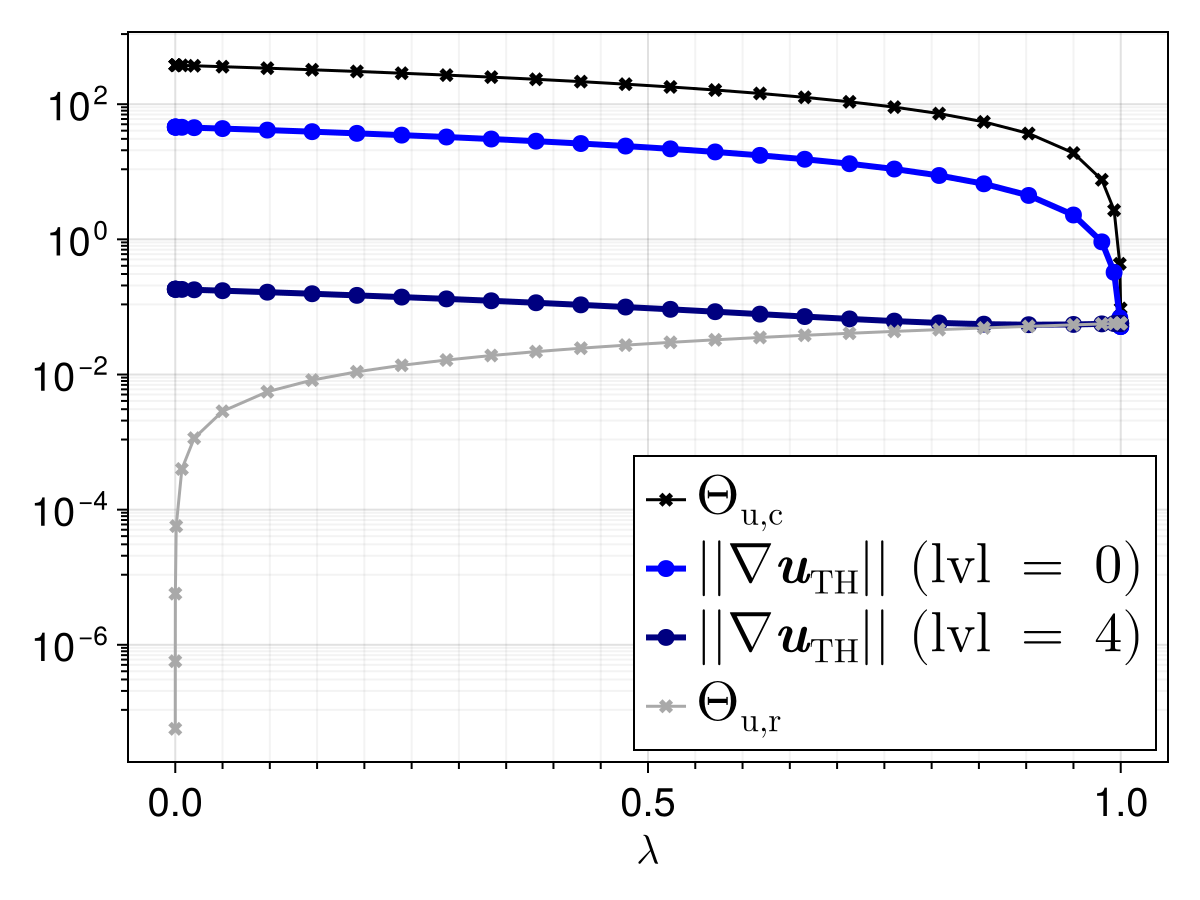}
\includegraphics[width=0.325\textwidth]{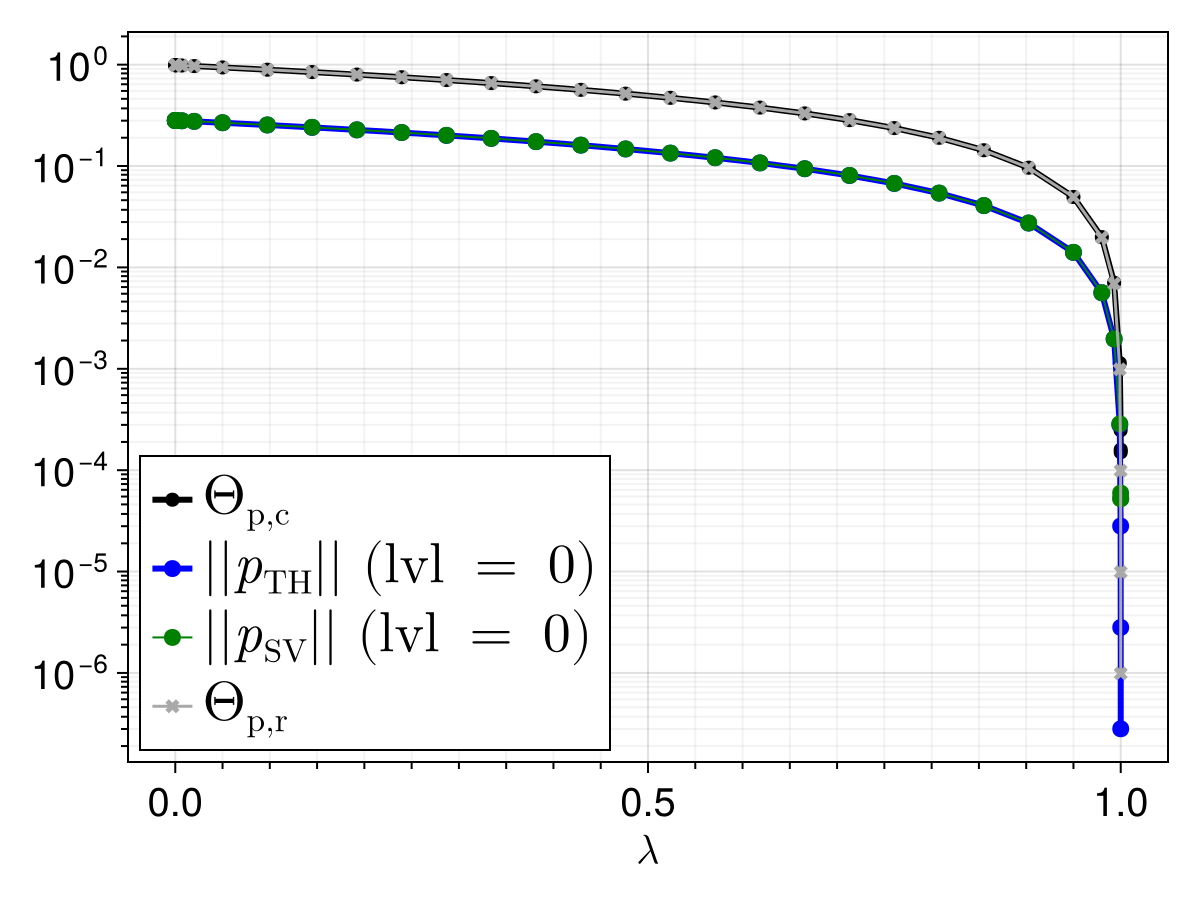}
\caption{\label{fig:example_stability_stokes}Bounds and norms of the Scott--Vogelius velocity $\bu_\text{SV}$ (left) and the Taylor--Hood velocity $\bu_\text{TH}$ (center) and of the pressures of both methods (right) versus $\lambda$ for $\mu = 10^{-3}$.}
\end{figure}

\subsection{Pressure-robustness and a consistency error}
Figure~\ref{fig:example_stability_stokes} shows, that the norms of the
Scott--Vogelius solution $\bu_{\text{SV}}$ respect the refined bound $\Theta_{u,r}$, and for $\lambda = 0$ one obtains $\bu_\text{SV} = \bzero$, consistent with
$\bfe \in K^0$ and $\bu = \bzero$ in this case.
Instead, the Taylor--Hood solution $\bu_\text{TH}$ is nonzero in the case $\lambda = 0$ and shows large violations of the refined bound $\Theta_{u,r}$.
Indeed, it holds $\bu_\text{TH} = \mathcal{O}(\mu^{-1})$.
This illustrates the missing $u$-equivalence of the Taylor--Hood finite element method
and reveals its consistency error.
In the language of mixed finite element methods, it suffers from a
parameter-dependent (and data-dependent)
locking phenomenon, although the Taylor--Hood
element is discrete inf-sup-stable and thus free of the
\emph{Poisson-locking phenomenon}. On the contrary,
the Scott--Vogelius element for the Stokes problem preserves the
$u$-equivalence and is free
of any locking phenomenon (on the used mesh family).
Following \cite{MR3683678}, the consistency error of the Taylor--Hood
element can also
be understood as a consistency error of the \emph{discrete vorticity equation} corresponding to the Taylor--Hood discretization.
A comparison of the norms of the Taylor--Hood solutions
on a very coarse mesh and on a (four times uniformly refined) finer mesh shows that they
move closer to the refined bound $\Theta_{u,r}$,
because the consistency error becomes smaller
on finer meshes.

The importance and the character
of this consistency error is underlined by
the next example, that considers the (nonlinear)
time-dependent Navier--Stokes problem
\begin{align*}
  \bu_t - \mu \Delta \bu + (\bu \cdot \nabla) \bu + \nabla p = \bzero
  \quad \text{and} \quad
  \mathrm{div} \bu = 0
\end{align*}
on the unit time interval and unit square
space domain $\Gamma \times \Omega := (0,1) \times (0,1)^2$
with inhomogeneous Dirichlet velocity boundary conditions
that match the initial velocity $\bu(0,x, y) := \nabla(x^3 - 3xy^2)$. It is the gradient of a harmonic polynomial and therefore in the class of potential flows of Section~\ref{sec:potential_flows}
with $(\bu \cdot \nabla) \bu = \frac{1}{2} \nabla \lvert \bu \rvert^2 \in K^0$ and hence $\| (\bu \cdot \nabla) \bu \|_{K^\prime} = 0$.
In other words, without exterior forces the exact solution $\bu$ stays constant in time.
Since, the initial velocity is quadratic, it is in both
the velocity ansatz spaces $V_h$ of the Taylor--Hood and
the Scott--Vogelius finite element methods.
The consistent Scott--Vogelius method preserves this velocity solution over time, since $\| (\bu \cdot \nabla) \bu \|_{K_h^\prime} = 0$. Thus, the Scott--Vogelius element `recognizes'
that the
potential flow under investigation is not \emph{convection-dominant},
since its discrete dynamics is connected to an appropriate semi norm,
in which the strength of the discrete nonlinear convection term is zero.
On the other hand, the Taylor--Hood solution $\bu_\text{TH}$ is spoiled by the consistency error
$\| (\bu \cdot \nabla) \bu \|_{K_h^\prime} \neq 0$.
Indeed, Figure~\ref{fig:example_potential_flow} tracks the $L^2$ velocity error
of both discretizations over time, with an implicit Euler time integration with timestep $0.01$, for $\mu = 10^{-4}$
on a given unstructured mesh. It can be seen that the error of the Scott--Vogelius method stays zero, while the error of the Taylor--Hood solution grows over time.

\begin{figure}
\centering
\includegraphics[width=0.55\textwidth]{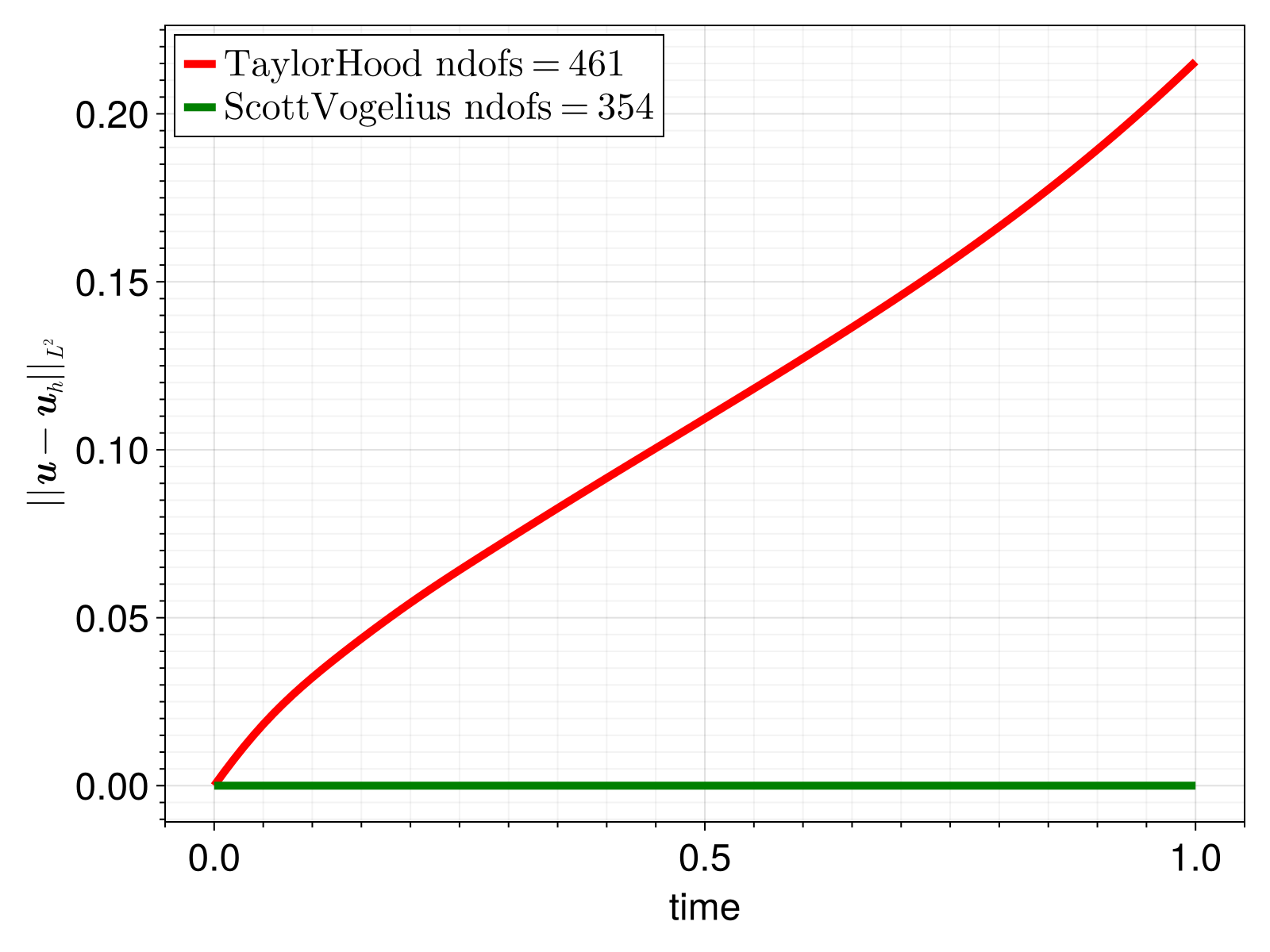}
\caption{\label{fig:example_potential_flow}Time-dependent error evolution for a quadratic potential flow of the Taylor--Hood and Scott--Vogelius finite element methods for the Navier--Stokes problem with $\mu = 10^{-4}$.}
\end{figure}

%%%%%%%%%%%%%%%%%
\section{Conclusions and Outlook}
\label{sec:conclusions}
The main message of this contribution is:
In mixed and related problems the data $f$ has a quality --- in the sense that there are two kinds of data. This quality can be
quantified by appropriate semi norms.
The refined investigation of stability estimates demonstrates
that respecting the quality of data leads to a more
predictive numerical analysis.
These semi norms are not only
relevant for the numerical analysis of data, but also for the analysis of involved operators
in more challenging problems:
indeed, in the Navier--Stokes problem it also allows to measure whether an incompressible flow is
\emph{convection-dominated}. Further, we connect the semi norms to physical regimes and thereby show the importance of the considerations for practical applications.
In future we will work on a
conceptual update for the numerical analysis of mixed finite element
methods by applying semi norms in a more abstract setting.

\bibliographystyle{siamplain}
\bibliography{literature}

\end{document}